\documentclass[12pt]{article}
\usepackage{amssymb, amsmath, amsthm, oldgerm, srcltx}
\usepackage{graphicx}
\usepackage{color}
\usepackage{titletoc}
\newcommand{\tc}{\textcolor{red}}

\newtheorem{thm}{Theorem}[section]
\newtheorem{lem}{Lemma}[section]
\newtheorem{cor}[lem]{Corollary}
\newtheorem{prop}[thm]{Proposition}
\newtheorem{rem}[thm]{Remark}
\numberwithin{equation}{section}

\newcommand{\abs}[1]{\left\vert#1\right\vert}

\newcommand{\E}{\mathbf{E}\,}

\newcommand{\R}{\mathbf{R}}

\newcommand{\re}{\mathrm{Re}\;\!}
\newcommand{\im}{\mathrm{Im}\;\!}
\newcommand{\Tr}{\mathrm{Tr}\;\!}

\newcommand{\lna}{l_{n,\alpha}}

\newenvironment{Proof of}{\removelastskip\par\medskip
\noindent{\em Proof of} \rm}{\penalty-20\null\hfill$\square$\par\medbreak}

\def\be{\begin{equation}}
\def\en{\end{equation}}
\def\bee{\begin{eqnarray*}}
\def\ene{\end{eqnarray*}}

\def\E{{\bf E}}

\def\s{\sigma_{n-1}}

\def\R{{\mathbb R}}

\def\Tr{{\rm Tr}\,}

\def\<{\left<}
\def\>{\right>}

\def\1{{\bf 1}}
\def\4{\kern1pt}

\begin{document}

\vspace{1in}

\title
{\bf On the Rate of Convergence to the  Marchenko--Pastur Distribution }

\vspace{2in}

\author{{\bf F. G\"otze}
\\{\small Faculty of Mathematics}
\\{\small University of Bielefeld}\\{\small Germany}
\and {\bf A. Tikhomirov}\\{\small Department of Mathematics
}\\{\small Komi Research Center of Ural Division of RAS,}\\{\small Syktyvkar state University}
\\{\small Syktyvkar, Russia}}
\date{}
\maketitle

\bibliographystyle{plain}

\maketitle

\footnote{ Partially supported by RFBR, grant  N~14-01-00500. 
Partially supported by the Program of UD of RAS N 12-P-1-1013
Partially supported by CRC 701 ``Spectral Structures and Topological
Methods in Mathematics'', Bielefeld.
}

\begin{abstract}
\hskip-6mm Let $\mathbf X=(X_{jk})$ denote an  $n\times p$ random matrix with
entries $X_{jk}$, which are independent for $1\le j\le n,1\le  k\le p$. We consider
the rate of convergence of the empirical  spectral distribution function of the
matrix $\mathbf W=\frac1p\mathbf X\mathbf X^*$ to the Marchenko--Pastur law.
We assume that $\mathbf E X_{jk}=0$,
$\E X_{jk}^2=1$ and that the distributions of the matrix elements $X_{jk}$ have a uniformly
sub exponential decay in the sense that there exists a constant $\varkappa>0$
such that for any $n,p\ge 1$ and  $1\le j \le n,\,1\le k\le p $ and any $t\ge 1$ we have\vspace{0.05in}
\newline$\qquad\qquad\qquad\qquad\qquad
\mathbf {Pr}\{|X_{jk}|>t\}\le \varkappa^{-1}\text{\rm exp}\{-t^{\varkappa}\}.
$\vspace{0.05in}\newline
By means of a recursion argument it is shown that the  Kolmogorov distance between the empirical spectral distribution of the sample covariance matrix $\mathbf W$
and the Marchenko--Pastur distribution is  of order $O(n^{-1}\log^{4+\frac4{\varkappa}} n)$  with high probability.
\end{abstract}


\section{Introduction}
\setcounter{equation}{0}

For any $n,p\ge1$, consider a family
of independent random variables $\{X_{jk},\, 1 \leq j \leq n, 1\leq k \leq p\}$, defined on some probability space
$(\Omega,{\textfrak M},\Pr)$. Let $\mathbf X=(X_{jk})$ be a matrix of order $n\times p$ and
let  $\mathbf W=\frac1p\mathbf X\mathbf X^*$. Denote by
$\{s_1^2,\dots,s_n^2\}$  the eigenvalues of the matrix $\mathbf W$ and introduce the
associated spectral distribution function
$$
\mathcal F_{n}(x) = \frac{1}{n}\ {\rm card}\,\{j \leq n: s_j^2 \leq
x\}, \quad x \in \R.
$$
Averaging over the random values $X_{ij}(\omega)$, define the expected
(non-random) empirical distribution functions $F_{n}(x) = \E\,\mathcal F_{n}(x)$.
We assume that $p=p(n)$ and $\lim_{n\to\infty}\frac np=y\in(0,\infty)$. Without loss of generality we shall assume that $y\in(0,1]$.
Let $G_y(x)$ denote the Marchenko--Pastur distribution function with density
$g_y(x)=G_y'(x)=\frac1{2yx\pi}\sqrt{(x-a)(b-x)}I_{[a,b]}(x)$, where $\mathbb I_{[a,b]}(x)$
denotes the  indicator--function of the interval $[a,b]$, $a=(1-\sqrt y)^2$, $b=(1+\sqrt y)^2$. We shall study the
rate of convergence
 $\mathcal F_{n}(x)$ to the Marchenko--Pastur law assuming that
\begin{equation}\label{exptails}
\Pr\{|X_{jk}|>t\}\le \varkappa^{-1}\exp\{-t^{\varkappa}\},
\end{equation}
 for some $\varkappa>0$ and any $t\ge1$.
 The rate of convergence to the Marchenko--Pastur law  has been  studied by several authors.
 In particular, the present authors
 proved in \cite{GT:04} that the Kolmogorov distance between $\mathcal F_n(x)$ and  the distribution function
  $G_y(x)$,
 $\Delta_n^*:=\sup_x|\mathcal F_n(x)-G_y(x)|$ is of order
 $O_P(n^{-\frac12})$.
Bai et al.  showed in \cite{Bai:03} that
 $\Delta_n:=\sup_x| F_n(x)-G_y(x)|=O(n^{-\frac12})$.
  For the Laguerre Unitary Ensemble G\"otze and Tikhomirov proved in \cite{GT:05}
that $\Delta_n=O(n^{-1})$.

Let $y=\frac np\in(0,1]$ in  what follows.
For any positive constants $\alpha>0$  and $\varkappa>0$ define the quantities
\begin{equation}\label{beta}
 \lna :=\log n\4(\log \log n)^{\alpha}
\quad \text{and} \quad
 \beta_n:=(\lna)^{\frac1{\varkappa}+1}.
\end{equation}
The main result of this paper is the following
\begin{thm}\label{main} Let $\E X_{jk}=0$, $\E X_{jk}^2=1$ and
there exists a constant $\varkappa>0$ such that for all $n,p\ge1$ any $1\le j\le n$ and $1\le k\le p$ and any $t\ge 1$,
condition \eqref{exptails} holds.
 Then for any $\alpha>0$ there exist a positive constants $C$ and $c$, depending on $\varkappa$, $\alpha$ and $y$ such that
\begin{equation}\label{kolmog}
\Pr\{\sup_x|\mathcal F_n(x)-G_y(x)|> n^{-1}\beta_n^4\}\le C\exp\{-c\4\lna\}.
\end{equation}
\end{thm}
\begin{rem}\label{localization} Let $\gamma_{nj}$ denote quantiles of order $\frac jn$ of distribution $G_y(x)$, $G_y(\gamma_{nj})=\frac jn.$ 
 Inequality \eqref{kolmog} implies that
 \begin{align}
\Pr\Bigl\{\, \exists\,\,j\in[c\beta_n^4,n-c\beta_n^4]:\,|s_j^2-\gamma_{nj}|\ge C\beta_n^4
\Big[\min\{j,N-j+1\}\Big]^{-\frac13}n^{-\frac23} \Bigr\} \notag\\ \le C\exp\{-c\lna\}.
\end{align}
\end{rem}
We apply the result of Theorem \ref{main} to investigation of eigenvectors of the matrix $\mathbf W$.
 Let $\mathbf u_j=(u_{j1},\ldots,u_{jn})^T$  be eigenvector of the matrix
$\mathbf W$ corresponding to eigenvalue $s_j^2$, $j=1,\ldots,n$.
We prove the following result.
\begin{thm}\label{eigenvector}Under the conditions of Theorem \ref{main} for any $\alpha>0$ 
there exist constants $C$, $c$, depending on $\varkappa$, 
$\alpha$ and $y$
such that
\begin{equation}\label{deloc}
 \Pr\{\max_{1\le j,k\le n}|u_{jk}|^2>\frac{\beta_n^4}n\}\le C\exp\{-c\4\lna\}
\end{equation}
and
\begin{equation}\label{deloc1}
\Pr\{\max_{1\le k\le n}\Big|\sum_{\nu=1}^k|u_{j\nu}|^2-\frac kn\Big|>\frac{\beta_n^2}{\sqrt n}\}\le C\exp\{-c\lna\}.
\end{equation}

\end{thm}

We shall investigate as well the Stieltjes transform of distribution of singular values of sample covariance matrices.
We shall use the following symmetrization of one-sided distribution function. Let $F(x)$ be the distribution function of non-negative random variable. Define its symmetrization as
\begin{equation}
\widetilde F(x)=\frac{1+\text{\rm sign}xF(x^2)}2.
\end{equation}
Let $s_F(z)=\int_0^{\infty}\frac1{x-z}dF(x)$ and $s_{\widetilde F}(z)=\int_{-\infty}^{\infty}\frac1{x-z}d\widetilde F(x)$
be Stieltjes transforms of distribution functions $F(x)$ and $\widetilde F(x)$ respectively.
It is straightforward to check that 
$$
\widetilde s(z)=zs(z^2).
$$

Note that symmetrization of Marchenko -- Pastur distribution function $\widetilde G_y(x)$ has the density
\begin{equation}\label{march}
\widetilde G'_y(x)=\frac1{2\pi x}\sqrt{(x^2-a)(b-x^2)}\mathbb I\{\sqrt a\le |x|\le \sqrt b\},
\end{equation}
and Stieltjes transform $S_y(z)$, satisfying the equation
\begin{equation}\label{reprst}
yS_y^2(z)+(z+\frac{y-1}z)S_y(z)+1.
\end{equation}
See Section 3 in \cite{GT:09}.
Let $m_n(z)$ and $\widetilde m_n(z)$ denote the Stiletjes transform of distribution function $F_n(x)$ and $\widetilde F_n(x)$ respectively.
Let 
\begin{equation}
\Lambda_n=\widetilde m_n(z)-S_y(z).
\end{equation}
Put the following quantities. For any $z=u+iv\in\mathbb C_+$ we define the function
$$
\gamma(z)=\min\{|1-\sqrt y-|u||,|1+\sqrt y-|u||\}.
$$
Let 
\begin{equation}\label{v0}
 v_0=A_0n^{-1}\beta_n^4
\end{equation}
with some positive constant $A_0$  which  will be chosen later. 
For any $\sqrt y>\varepsilon>0$, $V>v_0$ we introduce the region on the complex plane 
$$\mathbb G=\mathbb G(A_0,b,\varepsilon)=\{z=u+iv:\4 1-\sqrt y+\varepsilon\le |u|\le1+\sqrt y-\varepsilon,
v_0/\gamma(z)\le v\le V.\}$$ 
We prove the following estimations of $\Lambda_n$.
\begin{thm}\label{stlambda1}Under the conditions of Theorem \ref{main} there exist constants $C,c$ depending on $\varkappa$, $y$ and $\alpha$  such that for any $z=u+iV$
with $V=4\sqrt y$,
\begin{equation}
  \Pr\{|\Lambda_n(z)|\ge \frac{C\beta_n^3|S_y(z)|^{\frac32}}{n}\}\le \exp\{-c\lna\}.
 \end{equation}
\end{thm}
\begin{thm}\label{stlambda}Under the conditions of Theorem \ref{main} there exist constants $C,c$ depending on $\varkappa$, $y$ and $\alpha$ such that for any $z\in\mathbb G$
\begin{equation}
  \Pr\{|\Lambda_n(z)|\ge \frac{C\beta_n^3}{nv}\}\le \exp\{-c\lna\}.
 \end{equation}
\end{thm}

Similar results as in Theorems \ref{main} and \ref{eigenvector} for the sample covariance matrices with $y=1$ and assuming \eqref{exptails} with $\varkappa=2$
(sub-gaussian r.v.'s) were obtained recently in \cite{SchleinMaltsev}. Assuming that the entries of random matrix $\mathbf X$ have finite moments of all orders, Erd\"os and 
coauthors  in \cite{ErdosKnowles} proved the bound of the rate of convergence to the Marchenko--Pastur law of order $O(n^{-1+\varepsilon})$ for any $\varepsilon>0$.

The proof of Theorem \ref{main} is based on the estimation of the Kolmogorov distance between   
any distribution and a symmetrizing Marchenko--Pastur distribution  via their  Stieltjes transforms and on the Theorem \ref{stlambda}. 
We start by showing an improved bound for the Kolmogorov distance between   
any distribution and a symmetrizing Marchenko--Pastur distribution  using Stieltjes transforms. 
Compared  to a similar inequality in \cite{BGT:2010} 
we shall use here the 
vanishing of the  density at  the end points of support of the Marchenko--Pastur distribution. 
The  most  of the paper is devoted to the proof of Theorem \ref{stlambda}. To prove it we analyze the behavior of error terms in the representation \eqref{repr01}
of diagonal entries of the resolvent matrix.
We use here  a relatively short recursion argument based on the approach
developed in  \cite{GT:09}  and \cite{GT:11}   and  ideas similar to those used in  
Erd\"os, Yau and Yin \cite{ErdosYauYin:2010a}, Lemma  3.4.
This approach is based on the classical resolvent recursion
between random matrices of dimension $n$ an $n-1$, where
the error terms are stochastically bounded by means of
large deviation probability bounds for martingales.
Starting at fixed distance from the real line, where the Stieltjes-transforms are trivially bounded, we 
iterate these bounds using estimates for their derivatives
to a point nearer to the real line by a step of size $n^{-b}$. Assuming exponential moment conditions this recursion
of bounds for the  Stieltjes-transform  \eqref{repr03} can be iterated (approximately $n^b$-times) till reaching a distance
$n^{-1}(\log n)^{O(1)}$ to the real line, which finally
provides the final error bound.  
\newpage
\tableofcontents
\titlecontents{section}[1.5em]{\addvspace{1pt}}{\contentslabel{1.5em}}{}{\titlerule*[0.72pc]{.}\contentspage}
\newpage

\section{Kolmogorov distance between any distribution function and the symmetrizing 
Marchenko--Pastur distribution function}
To bound $\Delta_n^*$ we shall use an approach developed in
G\"otze and Tikhomirov \cite{GT:09} and \cite{GT:04}. We modify a bound for the  
Kolmogorov distance between
distribution functions based on their Stieltjes transforms obtained in \cite{GT:05}, Lemma 2.1.
Recall that $\widetilde  G_y(x)$ denotes the symmetrization of the distribution function $G_y(x)$ defined by the equality
\begin{equation}\label{symmetr}
\widetilde  G_y(x)=\frac{1+\text{\rm sign}\4x\4G_y(x^2)}2,
\end{equation}
For $y=1$ the distribution function $\widetilde  G_y(x)$ is the distribution function of 
the semi-circular law.
 Given $\frac{\sqrt y}2\ge\varepsilon>0$ introduce the interval $\mathbb J_{\varepsilon}=
[1-\sqrt y+\varepsilon,1+\sqrt y-\varepsilon]$ and
$\mathbb J'_{\varepsilon}=
[1-\sqrt y+\frac12\varepsilon,1+\sqrt y-\frac12\varepsilon]$.
For any $x$ such that $|x|\in[1-\sqrt y,1+\sqrt y]$, define
$\gamma=\gamma(x):=\min\{|x|-1+\sqrt y,1+\sqrt y-|x|\}$. Note that $0\le\gamma\le\sqrt y$.
For any $x:\,|x|\in\mathbb J_{\varepsilon}$, we have $\gamma\ge\varepsilon$, respectively, 
for any $x:\,|x|\in\mathbb J_{\varepsilon}'$,
 we have $\gamma\ge\frac12\varepsilon$.
For a  distribution function $F$ denote by $S_F(z)$ its Stieltjes transform,
$$
S_F(z)=\int_{-\infty}^{\infty}\frac1{x-z}dF(x).
$$
\begin{prop}\label{smoothing}Let $v>0$ and $H>0$ and $\varepsilon>0$ be positive numbers such 
that
\begin{equation}
 \tau=\frac1{\pi}\int_{|u|\le H}\frac1{u^2+1}du=\frac34,
\end{equation}
and
\begin{equation}\label{avcond}
 2vH\le \varepsilon^{\frac32}.
\end{equation}
If $\widetilde G_y$ denotes the distribution function of the symmetrizing (as in \eqref{symmetr})
Marchenko--Pastur law, and $F$ is any distribution function,
 there exist some absolute constants $C_1$, $C_2$, $C_3$ depending on $y$ only such that
\begin{align}
\Delta(F,\widetilde G_y)&:= \sup_{x}|F(x)-\widetilde G_y(x)|\notag\\&\le
2\sup_{x:|x|\in\mathbb J'_{\varepsilon}}\Big|\im\int_{-\infty}^x(S_F(u+i\frac v{\sqrt{\gamma}})
-S_{\widetilde G_y}(u+i\frac v{\sqrt{\gamma}}))du\Big|+C_1v
+C_2\varepsilon^{\frac32}
\end{align}
with $C_1=\begin{cases}\frac {2H^2\sqrt 3}{\pi^2\sqrt{y(1-\sqrt y)}}&\text{ if }0<y<1,\\ 
\frac{H^2}{\pi}&\text{ if }y=1,\end{cases}$ and 
$C_2=\begin{cases}\frac {4}{\pi \sqrt{y(1-\sqrt y)}}&\text{ if }0<y<1,\\ \frac1{\pi}
&\text{ if }y=1.\end{cases}$.

\end{prop}	
\begin{rem}
 \begin{equation}
  H=\text{\rm tg}\frac{3\pi}8=1+\sqrt 2.
 \end{equation}
 \begin{proof}
 For a proof of this remark see the Appendix, Subsection \ref{profofsmoothing}.
 \end{proof}

\end{rem}

\begin{cor}\label{Cauchy}
 Under the conditions  of Proposition \ref{smoothing}, for any $V>v\sqrt2/\sqrt{\varepsilon}$, the following 
inequality holds
\begin{align}
 \sup_{x\in\mathbb J'_{\varepsilon}}&\left|\int_{-\infty}^x(\im(S_F(u+iv')
-S_{G_y}(u+iv'))du\right|\notag\\&\le
\int_{-\infty}^{\infty}|S_F(u+iV)-S_{G_y}(u+iV)|du\notag\\&+
\sup_{x\in\mathbb J'_{\varepsilon}}\left|\int_{v'}^V\left(S_F(x+iu)-S_{G_y}(x+iu)\right)du
\right|,
\end{align}
where $\gamma=\min\{|x|-1+\sqrt y,\,1+\sqrt y-|x|\}$ and $v'=\frac v{\sqrt{\gamma}}\le V$, for $x\in\mathbb J'_{\varepsilon}$.
\end{cor}
\begin{proof}Let $x:\4|x|\in\mathbb J'_{\varepsilon}$ be fixed. Let 
$\gamma=\gamma(x)=\min\{|x|-1+\sqrt y,\,1+\sqrt y-|x|\}$.
 Set $z=u+iv'$,
$v'\le V$. Since the functions of $S_F(z)$ and $S_{\widetilde G_y}(z)$ are analytic in 
the upper half-plane, it is enough to use Cauchy's theorem. We can write
\begin{equation}
\int_{-\infty}^{x}\im(S_F(z)-S_{G_y}(z))du=\lim_{L\to\infty}\int_{-L}^x
(S_F(u+iv')-S_{G_y}(u+iv'))du,
\end{equation}
for $x\in \mathcal J'_{\varepsilon}$.
Since $v'=\frac v{\sqrt{\gamma}}\le \frac{\varepsilon}{2H}$, without loss of generality 
we may assume that $v'\le 2$.
By Cauchy's integral formula, we have
\begin{align}
 \int_{-L}^x(S_F(z)-S_{\widetilde G_y}(z))du&=\int_{-L}^x(S_F(u+iV)
-S_{\widetilde G_y}(u+iV))du\notag\\&
+\int_{v'}^V(S_F(-L+iu)-S_{\widetilde G_y}(-L+iu))du\notag\\&-\int_{v'}^V(S_F(x+iu)
-S_{\widetilde G_y}(x+iu))du.
\end{align}
Denote by $\xi$ (resp. $\eta$) a random variable with distribution function $F(x)$ 
(resp. $\widetilde G_y(x)$). Then we have
\begin{equation}
 |S_F(-L+iv')|=\left|\E\frac1{\xi+L-iv'}\right|\le {v'}^{-1}\Pr\{|\xi|>L/2\}+\frac2L.
\end{equation}
Similarly,
\begin{equation}
 |S_{\widetilde G_y}(-L+iv')|\le {v'}^{-1}\Pr\{|\eta|>L/2\}+\frac2L.
\end{equation}
These inequalities imply that
\begin{equation}
\left|\int_{v'}^V(S_F(-L+iu)-S_{G_y}(-L+iu))du\right|\to 0\quad\text{as}\quad L\to\infty,
\end{equation}
which completes the proof.
\end{proof}
Combining the results of Proposition \ref{smoothing} and Corollary \ref{Cauchy}, we get
\begin{cor}\label{smoothing1}
 Under the conditions of Proposition \ref{smoothing} the following inequality holds
\begin{align}
 \Delta(F,\widetilde G_y)&\le 2\int_{-\infty}^{\infty}|S_F(u+iV)-S_{\widetilde G_y}(u+iV)|du
+C_1v+C_2\varepsilon^{\frac32}\notag\\&
  +2\sup_{x\in\mathbb J'_{\varepsilon}}\int_{v'}^V|S_F(x+iu)-S_{\widetilde G_y}(x+iu)|du,
\end{align}
where $v'=\frac v{\sqrt{\gamma}}$ with $\gamma=\min\{|x|-1+\sqrt y,\,1+\sqrt y-|x|\}$.

\end{cor}

 \section{Truncation}We consider the truncated random variables $\widehat X_{jl}$ defined by
\begin{equation}\label{trunc000}
 \widehat X_{jl}:=X_{jl}\mathbb I\{|X_{jl}|\le c\lna^{\frac1{\varkappa}} \}-\E X_{jl}\mathbb I\{|X_{jl}|\le c\lna^{\frac1{\varkappa}}\}.
\end{equation}
Let $\widehat{\mathbf  X}=(\widehat X_{jk})$ and 
$$
\widehat V=\frac1{\sqrt p}\begin{bmatrix}&\mathbf O&\widehat{\mathbf X}\\&{\widehat{\mathbf X}}^*&\mathbf O\end{bmatrix}.
$$
Let $\widehat {\mathcal F}_n(x)$ denote the empirical spectral distribution function of the matrix $\widehat{\mathbf V}$ and let $\widehat m_n(z)$
be the corresponding Stieltjes transform. Let $\sigma_{jk}^2=\E(\widehat X_{jk})^2$. Introduce the  r.v.'s $\widetilde X_{jk}=\sigma_{jk}^{-1}\widehat X_{jk}$.
Consider the matrix $\widetilde {\mathbf X}=(\widetilde X_{jk})$ and
$$
\widetilde{\mathbf V}=\frac1{\sqrt p}\begin{bmatrix}&\mathbf O&\widetilde{\mathbf X}\\&{\widetilde{\mathbf X}}^*&\mathbf O\end{bmatrix}.
$$
Let $\widetilde{\mathcal  F}_n(x)$ and $\widetilde m_n(z)$ denote the corresponding empirical spectral distribution function of the matrix $\widehat{\mathbf V}$ and its  Stieltjes  transform respectively.
 Let $\widetilde{\mathcal  F}_n(x)$ and $\widetilde m_n(z)$ denote the corresponding empirical spectral distribution function of the matrix $\widehat{\mathbf V}$ and its  Stieltjes  transform respectively.
\begin{lem}\label{truncnew}Under conditions of Theorem \ref{main} there exist  constants $C,c$ depending on $\varkappa$  only such that
\begin{align}
\Pr\{|\int_{-\infty}^{\infty}|m_n(u+iV)-\widehat m_n(u+iV)|du\ge \frac{C}{n}\}\le\exp\{-c\lna\}.
\end{align}
\end{lem}
\begin{proof}First we note that for all resolvent matrices $\mathbf R(u+iV)$ 
\begin{align}
\int_{-\infty}^{\infty}\|\mathbf R(u+iV)\|_2^2du=\sum_{j=1}^n\int_{-\infty}^{\infty}\frac1{(\lambda_j-u)^2+V^2}du=
\frac{n\pi}V.
\end{align}
By resolvent equality, we have
\begin{align}
m_n(z)-\widehat m_n(z)=\frac1n\Tr(\mathbf V-\widehat{\mathbf V})\mathbf R\widehat{\mathbf R}. 
\end{align}
Applying Cauchy - Schwartz inequality, we get
\begin{align}
|m_n(z)-\widehat m_n(z)|\le \frac1n\|\mathbf V-\widehat{\mathbf V}\|_2(\|\mathbf R\|_2^2+\|\widehat{\mathbf R}\|_2^2).
\end{align}
After integrating we arrive
\begin{align}
\int_{-\infty}^{\infty}|m_n(u+iV)-\widehat m_n(u+iV)|du\le C\|\mathbf V-\widehat{\mathbf V}\|_2,
\end{align}
for some numerical constant $C$.
Furthermore, by definition of matrices $\mathbf V$ and $\widehat{\mathbf V}$ we have
\begin{align}
\|\mathbf V-\widehat{\mathbf V}\|_2^2\le 4\|\mathbf X-\widehat{\mathbf X}\|_2^2\le\frac1n\sum_{j=1}^n\sum_{k=1}^p|X_{jk}-\widehat X_{jk}|^2+\frac1n\sum_{j=1}^n\sum_{k=1}^p|\E\widehat X_{jk}|^2
\end{align}
From condition \eqref{exptails} it follows now 
\begin{equation}
\Pr\{\|\mathbf V-\widehat{\mathbf V}\|_2\ge \frac C{n^2}\}\le \exp\{-c\lna\}.
\end{equation}
Thus Lemma \ref{truncnew} is proved.
\end{proof}
\begin{lem}\label{trunc}
 Assuming the conditions of Theorem \ref{main} there exist constants $C, c>0$ such that for all $z\in\mathbb G$
 \begin{equation}\notag
  \Pr\{|m_n(z)-\widetilde m_n(z)|\ge \frac{c}{n^2v^2}\}\le \exp\{-c\lna\}.
 \end{equation}

\end{lem}
\begin{proof}For a  proof of this Lemma see Subsection \ref{proofoftrunc} of the Appendix.
 
\end{proof}
\begin{rem}\label{trunc00}
 In what follows we shall assume that r.v.'s $X_{jl}$ satisfy the condition
 \begin{equation}\label{trunc2}
  |X_{jl}|\le C\lna^{\frac1{\varkappa}}, \quad\E X_{jl}=0 \quad\text{  and   }\quad \E X_{jk}^2=1.
 \end{equation}
We shall omit the symbol $\,\,\widehat{\cdot}\,\,$ in the notation of the truncated r.v.'s and corresponding characteristics of truncated matrices.
\end{rem}
\section{Diagonal entries of resolvent matrices}\label{diag} In this Section we investigate the diagonal entries of resolvent matrix $R_{jj}$.
Let
$v_k=v_0+\frac {k\sqrt y}{n^2}$, where $k=0,\ldots,N$ and $N=4n^2$. We introduce the following events
\begin{align}
\mathcal A_k&=\{|m_n(z)-s(_y(z)|\le\frac1{2\sqrt y}, \text{ for }z=u+iv, v\ge v_k\},\notag\\
\mathcal B_k&=\{|\varepsilon_j|\le {\gamma_0}{\sqrt y},\text{ for }j=1,\ldots,n;
\,z=u+iv, v\ge v_k\}, 
\end{align}
where $\gamma_0$ is some absolute constant and is to be chosen later.
\subsection{The Key Lemma}
\begin{lem}\label{key}Under the conditions of  Theorem \ref{main} there exist absolute constants $\gamma_0$ and $c_0$ such that, for $k=1,\ldots,N$, 
\begin{equation}
\mathcal A_{k}\cap\mathcal B_k\subset\mathcal A_{k-1}.
\end{equation}
\end{lem}
\begin{proof}The proof of this Lemma see in Appendix, Subsection \ref{profofkey}.
\end{proof}
\begin{cor}\label{cor001}
 We have
 \begin{equation}\notag
  \Pr\{\mathcal A_k^c\}\le \sum_{l=k+1}^N\Pr\{\mathcal B_l^c\cap \mathcal A_l\}.
 \end{equation}
\end{cor}
\begin{proof}Note that $v_N\ge 4\sqrt y$. We have, for $z=u+iv$ with $v\ge v_N$
$$
\max\{|m_n(z)|,|s(z)|\}\le \frac1{4\sqrt y} \text{ a. s.}
$$
That means that
\begin{equation}
 \Pr\{A_N\}=1.
\end{equation}
By Lemma \ref{key}, we have
\begin{equation}\label{cru6}
 \Pr\{\mathcal A_k^{(c)}\}\le \Pr\{(\mathcal B_{k+1}\cap\mathcal A_{k+1})^c\}\le \Pr\{\mathcal B_{k+1}^c\cap\mathcal A_{k+1}\}+\Pr\{\mathcal A_{k+1}^c\}.
\end{equation}
The claim of Corollary \ref{cor001} now follows from \eqref{cru6} by induction.
 \end{proof}
\begin{cor}
For $k=0,\ldots, N$ the following inequality holds
\begin{equation}
\Pr\{\mathcal A_k^c\}\le \sum_{l=k+1}^N\Pr\{\mathcal B_l^c\}.
\end{equation}
\end{cor}

\begin{lem}\label{cru2}
 Under the conditions of  Theorem \ref{main} there exist  constant $A_0$  such that, for $k=1,\ldots,N$,  
\begin{equation}
\Pr\{\mathcal A_{k}\cap\mathcal B_k^c\}\le \exp\{-c\lna^2\}.
\end{equation}
\end{lem}
\begin{proof}
 Note that for $\omega\in\mathcal A_k$
 \begin{equation}
  \im m_n^{j)}(z)\le \frac 3{2\sqrt y}+\frac1{nv}.
 \end{equation}
Furthermore, take the constant $A_0$ in the definition of $v_0$ such that
\begin{equation}
 A_0\ge \frac{3C^2}{2y\sqrt y\gamma_0^2}+\frac{2C^2}{y\gamma_0^2},
\end{equation}
where $C$ is the constant from Proposition \ref{epsilon}.
Then it is straightforward to check that, for all $v\ge v_0$,
\begin{equation}
 \gamma_0\sqrt y\ge \frac{C\sqrt{\im m_n^{(j)}(z)+\frac1{nv}}\lna^{2+\frac2{\varkappa}}}{\sqrt{nv}}.
\end{equation}
We may write now that, for any $k=1,\ldots,N$,
\begin{equation}
 \Pr\{\mathcal A_{k}\cap{\mathcal B_k}^c\}\le \Pr\{|\varepsilon_j|\ge \frac{C\sqrt{\im m_n^{(j)}(z)+\frac1{nv}}\lna^{2+\frac2{\varkappa}}}{\sqrt{nv}}\}.
\end{equation}
Applying now Proposition \ref{epsilon} we get the claim.
Thus lemma \ref{cru2} is proved.
\end{proof}


Introduce now for $z=u+iv\in\mathbb C_+$ and $v'\ge v_0$,
\begin{equation}\notag
\mathcal A_{v'}=\{|m_n(z)-s(z)|\le \frac1{2\sqrt y}, \text{ for all }
u\in\mathbb R, \text{ and }v\ge v'\}.
\end{equation}
 Applying Lemmas \ref{key} and \ref{cru2}, we get
\begin{cor}\label{cru3} The following inequality holds, 
\begin{equation}\label{gnu1}
 \Pr\{\mathcal A_{v_0}^c\}\le \exp\{-C\lna^2\},
\end{equation}
for some positive constant $C>0$.
There exist a constant $C>0$ such that
\begin{equation}\label{rjj}
\Pr\{\{|R_{jj}|\le \frac3{2\sqrt y},\text{ for all } j=1,\ldots,n\}\}\ge1-
\exp\{-C\lna^2\}.
\end{equation}
Moreover, for $z=u+iv$ with $a\le |u|\le b$ and $0<v\le 4\sqrt y$, there exist a constant $\delta>0$ depending on $y$ only such that
\begin{equation}\label{rjj1}
\Pr\{\{|R_{jj}|\ge \delta,\text{ for any } j=1,\ldots,n\}\cap \mathcal U\}\ge1-
\exp\{-C\lna^2\}.
\end{equation}
We may take $\delta=\frac1{28}$.
\end{cor}
\begin{proof}
By Lemmas \ref{key} and \ref{cru2}, we have
\begin{align}\notag
\Pr\{\mathcal A_{v_0}^c\}&\le \sum_{t=1}^N\Pr\{\mathcal B_t^c\cap A_t\}\le
n^2\exp\{-C\lna^2\}\le \exp\{-C'\lna^2\}
\end{align}
with some positive constant $C'$.
This inequality and Lemma \ref{cru2} yield  inequality \eqref{gnu1}.
Furthermore, we note that
the events $\mathcal A_{v_0}$ and $\mathcal B_k$ together imply by \eqref{repr01}  that
\begin{equation}
|R_{jj}(u+iv)|\le \frac3{2\sqrt y}.
\end{equation}
Note that the events $\mathcal A_{v_0}$ and $\mathcal B_k$ imply by \eqref{repr01}
\begin{equation}\notag
 |R_{jj}(u+iv)|\ge \frac1{2|z+\frac{y-1}z+ym_n(z)|}\ge \frac1{4(1+4\sqrt y)}\ge \frac1{20},
\end{equation}
for $1-\sqrt y<|u|\le 1+\sqrt y$. We use here that, for $1-\sqrt y<|u|\le 1+\sqrt y$,
$$
|z+\frac{y-1}z|\le 1+\sqrt y+|z|\le 2(1+\sqrt y)+v\le 2+6\sqrt y.
$$

This proves inequalities \eqref{rjj} and \eqref{rjj1}.
\end{proof}
 
\begin{prop}\label{rjj8}Under the conditions of  Theorem \ref{main} there exist  constants $C$  and $A$ depending on $\varkappa$  and $\alpha$ only such that, for $j=1,\ldots,n$ and $q\le A\log n$,  and for $v\ge v_0$
\begin{equation}
 \E|R_{jj}|^q\le C^q.
\end{equation}

\end{prop}
\begin{proof}
 Since $|R_{jj}|\le v^{-1}\le v_0^{-1}$ we have
\begin{equation}
   \E|R_{jj}|^q\le (\frac 3{2\sqrt y})^q+v_0^{-q}\Pr\{|R_{jj}|>\frac 3{2\sqrt y}\}.
 \end{equation}
Applying Corollary \ref{cru3}, inequality \eqref{rjj}, we get 
\begin{equation}
   \E|R_{jj}|^q\le (\frac 3{2\sqrt y})^q+\exp\{c\log n^2-\lna^2\}\le C^q,                                             
\end{equation}
for some constant $C> \frac 3{2\sqrt y}$.
Thus Proposition \ref{rjj8} is proved.
\end{proof}

\begin{prop}\label{an}Under the conditions of  Theorem \ref{main} there exist  constants $C$  and $A$ depending on  $\varkappa$, $y$ and $\alpha$ only such that, for $j=1,\ldots,n$ and $q\le A\log n$,  and for $v\ge v_0$
\begin{equation}
 \E\frac1{|a_n(z)|^q}\le C^q.
\end{equation}

\end{prop}
\begin{proof}First we note that the even $\mathcal A_{v_0}$ yields that
\begin{equation}
 |a_n(z)|\ge |a(z)|-y|m_n(z)-S_y(z)|\frac{\sqrt y}{2}
\end{equation}

 Since $|a_n(z)|\ge \im z\ge v$ we have
\begin{equation}
   \E\frac1{|a_n(z)|^q}\le (\frac 2{\sqrt y})^q+v_0^{-q}\Pr\{|a_n(z)|>\frac {\sqrt y}2\}.
 \end{equation}
Applying Corollary \ref{cru3},  we get 
\begin{equation}
   \E|a_n(z)|^{-q}\le (\frac 2{\sqrt y})^q+\exp\{c\log n^2-\lna^2\}\le C^q,                                             
\end{equation}
for some constant $C\ge \frac 2{\sqrt y}$.
Thus Proposition \ref{an} is proved.
\end{proof}


\section{Proof of Theorem \ref{main}}To conclude the proof of Theorem \ref{main}
 we shall now apply the result of Corollary \ref{smoothing1} with $ v_0=\frac{A_0\beta_n^4}{n}$ and $V=4\sqrt y$
 to the empirical spectral
 distribution function
$\mathcal F_n(x)$ of the random matrix $\mathbf W$.
 At first we estimate the integral over the line  $V=4\sqrt y$. Consider  the integral
\begin{equation}\notag
 Int(V)=\int_{-\infty}^{\infty}|m_n(u+iV)-S_y(u+iV)|du
\end{equation}
for $V=4\sqrt y$.
Using Theorem \ref{stlambda1}, we have
\begin{align}\notag
\Pr\{ |Int(V)|&\le \frac{ C\beta_n^3}n\int_{-\infty}^{\infty}|S_y(u+iV)|^{\frac32}du\}\ge1- \exp\{-c\lna\}.
\end{align}
Finally, we note that
\begin{equation}\label{finish6}
 \int_{-\infty}^{\infty}|S_y(z)|^{\frac32}dx\le \int_{-\infty}^{\infty}\int_{-\infty}^{\infty}\frac1{((x-u)^2+V^2)^{\frac32}}du\,dG(x)\le C.
\end{equation}
The last inequality implies
\begin{equation}\label{final!}
\Pr\{ \int_{-\infty}^{\infty}|\Lambda_n(u+iV)|du> \frac {C\beta_n^4}n\}\le\exp\{-c\lna\}.
\end{equation}
Consider now $u\in\mathbb J_{\varepsilon}$, where $\varepsilon=v_0^{\frac23}$. By Proposition \ref{stlambda},
we have
\begin{equation}\label{fin10}
 \Pr\{|\Lambda_n(z)|\le \frac{\lna^{2+\frac2{\varkappa}}}{nv}\}\ge
 1-\exp\{-c\lna\}.
\end{equation}
Integrating now in $v\in [v_0/\sqrt{\gamma}, V]$, we get
\begin{equation}
 \int_{v_0/\sqrt{\gamma}}^V|\Lambda_n(u+iv)|dv\le
 \frac{C\beta_n^2\log n}{n}.
\end{equation}
Thus, Theorem \ref{main} is proved.
 \section{The proof of Theorem \ref{stlambda1}}
 We shall use the ``symmetrization'' of the spectrum sample covariance matrix as in  
\cite{GT:09}.
Introduce the $(p+n)\times (p+n)$ matrix
\begin{equation}
 \mathbf V=\frac1{\sqrt p}\begin{bmatrix}& \mathbf O\quad&\mathbf X\\&\mathbf X^*
\quad&\mathbf O\end{bmatrix},
\end{equation}
where $\mathbf O$ denotes a matrix of corresponding dimension  with zero entries.
Note that the eigenvalues of the matrix $\mathbf V$ are $\pm s_1,\ldots,\pm s_n,$ and $0$ 
with multiplicity $p-n$.
Let
$\mathbf R=\mathbf R(z)$ denote the resolvent matrix of  $\mathbf V$ defined by the equality
$$
\mathbf R=(\mathbf V-z\mathbf I_{n+p})^{-1},
$$
for all $z=u+iv$ with $v\ne 0$. Here and in what follows $\mathbf I_k$ denotes the identity 
matrix of order $k$.
Sometimes we shall omit the  sub index in the notation of the identity  matrix.

We shall use in what follows the representation, for $j=1,\ldots, n$,
\be\label{diagres}
R_{jj}=\frac1{-z-\frac1p{{\sum_{k,l=1}^p}}X_{jk}X_{jl}R^{(j)}_{k+n,l+n}}
\en
(see,
for example, Section 3 in \cite{GT:09}). We may rewrite it as
follows
\be\label{repr01}
 R_{jj}=-\frac1{z+ym_n(z)+\frac{y-1}z}+
\frac1{z+ym_n(z)+\frac{y-1}z}\varepsilon_jR_{jj}, \en
where
$\varepsilon_j=\varepsilon_{j1}+\varepsilon_{j2}+\varepsilon_{j3}$ and
\begin{align}
\varepsilon_{j1}&:=\frac1p{\sum_{k=1}^p}(X_{jk}^2-1)R^{(j)}_{k+n,k+n},\quad
\varepsilon_{j2}:=\frac1p{\sum_{1\le k\ne
l\le p}}X_{jk}X_{jl}R^{(j)}_{k+n,l+n},\notag\\
\varepsilon_{j3}&:=\frac1p\Bigl(\sum_{l=1}^p R^{(j)}_{l+n,l+n}-\sum_{l=1}^p R_{l+n,l+n}\Bigr).
\notag
\end{align}
This relation immediately implies the following  equations
\begin{align}\label{repr03}
 m_n(z)&=-\frac1{z+ym_n(z)+\frac{y-1}z}+\frac1{(z+ym_n(z)+\frac{y-1}z)}
\frac1n\sum_{j=1}^n\varepsilon_jR_{jj}
\end{align}
Introduce  notations  $a_n(z)=z+\frac{y-1}z+ym_n(z)$, and 
 $b_n(z)=a_n(z)+yS_y(z)$. 
Equality \eqref{reprst}  implies that
\begin{equation}\label{semi20}
 1-\frac y{(z+\frac{y-1}z+yS_y(z))a_n(z)}=1+\frac{yS_y(z)}{a_n(z)}=\frac{b_n(z)}{a_n(z)}.
\end{equation}
The representation \eqref{repr03}  implies 
\begin{equation}\label{repr100*} 
 \Lambda_n(z)=\frac{y\Lambda_n(z)}{(z+\frac{y-1}z+yS_y(z))a_n(z)}+\frac{T_n(z)}{a_n(z)}, 
\end{equation} 
where
\begin{align}
T_n(z):=\frac1{n}\sum_{j=1}^n\varepsilon_jR_{jj}.
\end{align}
From here it follows by solving for $\Lambda_n(z)$ that
\begin{equation}\label{lambda'}
 \Lambda_n(z)=\frac{T_n(z)}{b_n(z)}.
\end{equation}

We start from following
\begin{prop}\label{lambdalarge}Under conditions of Theorem \ref{main} there exist constant $A$ and $C$ such that for any $q\le A\log n$
for $V=4\sqrt y$,
\begin{equation}
\E|\Lambda_n(z)|^q\le \frac{Cq^{2q}|S_y(z)|^{\frac{3q}2}}{n^q}
\end{equation}
\end{prop}

\begin{proof}We start from the simple relations for the functions $a_n(z)$ and $b_n(z)$. Let
$a(z)=z+\frac{y-1}z+yS_y(z)=-\frac1{S_y(z)}$ and $b(z)=a(z)+yS_y(z)$. For $V=4\sqrt y$ we have
\begin{align}\label{re1*}
\max\{|m_n(z)|,|S_y(z)|\}&\le \frac1{4\sqrt y}, \quad |m_n(z)-S_y(z)|\le \frac1{2\sqrt y},\notag\\
|a_n(z)-a(z)|&\le \frac{\sqrt y}2\le \frac12|a(z)|, |a_n(z)|\ge \frac12|a(z)|=\frac1{2|S_y(z)|},\notag\\
|b_n(z)|&\ge \frac12|a(z)|=\frac1{2|S_y(z)|}.
\end{align}
Let
\begin{equation}\notag
 \varphi(z)=\overline z|z|^{q-2}.
\end{equation}
Using  equality \eqref{lambda'}, we may write, for $q\ge2$,
\begin{align}\notag
 \E|\Lambda_n|^q=\sum_{\nu=1}^3\E\frac{T_{n\nu}}{b_n(z)}\varphi(\Lambda_n).
\end{align}
where
\begin{equation}\notag
 T_{n\nu}=\frac1n\sum_{j=1}^n\varepsilon_{j\nu}R_{jj}.
\end{equation}
We consider first the term with $\nu=3$.
equality \eqref{e3} in the Appendix yields
\begin{equation}\label{3.9}
 \frac1n\sum_{j=1}^n\varepsilon_{j3}R_{jj}=-\frac12\frac d{dz}m_n(z)+\frac1zm_n(z),
\end{equation}
which implies that
\begin{equation}\label{nu3}
|\frac1n\sum_{j=1}^n\varepsilon_{j3}R_{jj}|\le \frac1V|m_n(z)|.
\end{equation}
Inequalities
\eqref{nu3},   \eqref{re1*} and representation \eqref{repr100*} together imply
\begin{align}\label{finish1}
 |\frac{T_{n3}}{b_n(z)}|&\le \frac1n|\sum_{j=1}^n\frac{\varepsilon_{j3}R_{jj}}{b_n(z)}|\le \frac {C|S_y(z)|}n|m_n(z)|\notag\\&\le \frac Cn|S_y(z)|^2(1+|T_n|).
\end{align}
Therefore, by H\"older's inequality,
\begin{align}
 \Big|\E\frac{T_{n3}\varphi(\Lambda_n)}{b_n(z)}\Big|\le \frac{C|S_y(z)|^{2}}n(1+\E^{\frac1q}|T_n|^q)
 \E^{\frac{q-1}q}|\varphi(\Lambda_n)|^{\frac{q}{q-1}}.\notag
\end{align}
It is straightforward to check that, by the Cauchy -- Schwartz inequality and  $|R_{jj}|\le V^{-1}$,
\begin{equation}\notag
 \E|T_n|^q\le \E\Big(\frac1n\sum_{j=1}^n|\varepsilon_j|^2\Big)^{\frac q2}
 \Big(\frac1n\sum_{j=1}^n|R_{jj}|^2\Big)^{\frac q2}
 \le \frac1{(4\sqrt y)^q}\E^{\frac12}\Big(\frac1n\sum_{j=1}^n|\varepsilon_j|^2\Big)^{q}.
\end{equation}
Applying Corollary \ref{eps1} and Lemmas \ref{eps2} and \ref{eps3} in the Appendix, we get that there exists an absolute  constant $C''>0$ 
such that for $q\le C'\log n$,
\begin{align}\notag
 \E^{\frac1q}|T_n|^q\le Cqn^{-\frac 12}\le C''CC'\le C.
\end{align}
Therefore,
\begin{equation}\notag
 \Big|\E\frac{T_{n3}\varphi(\Lambda_n)}{b_n(z)}\Big|\le \frac{C|S_y(z)|^{2}}n
 \E^{\frac{q-1}q}|\varphi(\Lambda_n)|^{\frac{q}{q-1}}.
\end{equation}
Furthermore, we represent, for $\nu=1,2$
\begin{align}\label{nu12}
 \frac1n\E\frac{\sum_{j=1}^n\varepsilon_{j\nu}
 R_{jj}\varphi(\Lambda_n)}{b_n(z)}=H_1+H_2,
\end{align}
where
\begin{align}\label{finish01}
 H_1&:=\frac1n\E\frac{\sum_{j=1}^n\varepsilon_{j\nu}
S_y(z)\varphi(\Lambda_n)}{b_n(z)},\notag\\
H_2&:=\frac1n\E\frac{\sum_{j=1}^n\varepsilon_{j\nu}(R_{jj}-
S_y(z))\varphi(\Lambda_n)}{b_n(z)}.
\end{align}
First we bound $H_2$. Applying   Cauchy -- Schwartz inequality followed by H\"older's inequality, we get
\begin{align}\label{eps0}
 |H_2|\le {C|S_y(z)|}\E^{\frac1{2q}}\Big(\frac1n\sum_{j=1}^n|\varepsilon_{j\nu}|^{2}\Big)^q\E^{\frac1{2q}}\Big(\frac1n\sum_{j=1}^n|R_{jj}-S_y(z)|^2\Big)^{q}
\E^{\frac{q-1}q}|\varphi(\Lambda_n)|^{\frac q{q-1}}.
\end{align}
Using the representation \eqref{repr01},  we may write
\begin{align}\label{lambda''}
 R_{jj}=S_y(z)-S_y(z)\varepsilon_jR_{jj}-S_y(z)\Lambda_nR_{jj}.
\end{align}
Applying the representations \eqref{lambda''} and \eqref{lambda'}, we obtain
\begin{align}\label{eps**}
 \frac1n\sum_{j=1}^n|R_{jj}(z)-S_y(z)|^{2}\le C|S_y(z)|^{2}\Big(\frac1n\sum_{l=1}^n|\varepsilon_l|^{2}\Big).
\end{align}
Combining inequalities \eqref{eps0} and \eqref{eps**}, we get
\begin{align}\notag
 |H_2|\le C|S_y(z)|^2\E^{\frac1q}\Big(\frac1n\sum_{j=1}^n|\varepsilon_{j}|^2\Big)^{q}\E^{\frac{q-1}q}|\varphi(\Lambda_n)|^{\frac q{q-1}}.
\end{align}
Applying now Corollary \ref{eps1} and Lemmas \ref{eps2}, \ref{eps3} in the Appendix, we get
\begin{align}\label{finish2}
 |H_2|\le \frac{C|S_y(z)|^{2}}{n}\E^{\frac{q-1}q}|\varphi(\Lambda_n)|^{\frac q{q-1}}.
\end{align}
Let 
\begin{align}\label{jate}
\eta_{j0}&:=\frac1p\sum_{l=1}^p[(\mathbf R^{(j)})^2]_{l+n,l+n},\notag\\
\eta_{j1}&:=\frac1p\sum_{l=1}^p(X_{jl}^2-1)[(\mathbf R^{(j)})^2]_{l+n,l+n},\notag\\
\eta_{j2}&:=\frac1p\sum_{1\le l\ne k\le p}X_{jl}X_{jk}[(\mathbf R^{(j)})^2]_{k+n,l+n},\notag\\
\eta_j&:=\eta_{j1}+\eta_{j2}.
\end{align}
Furthermore, introduce the notations
\begin{align}\label{amba}
 \Lambda_n^{(j)}&:=m_n^{(j)}(z)-s_y(z),\quad {\widetilde {\Lambda}}^{(j)}_n=\Lambda_n^{(j)}+\frac{s_y(z)}{2n}(1+\eta_{j0})+\frac1{2nz},\notag\\
a_n^{(j)}(z)&:=a(z)+y\Lambda_n^{(j)}=z+\frac{y-1}z+ym_n^{(j)}(z), \notag\\b_n^{(j)}(z)&:=b(z)+y\Lambda_n^{(j)}=z+\frac{y-1}z+y(m_n^{(j)}(z)+s_y(z)).
\end{align}
Note that
\begin{equation}
 \Lambda_n-\widetilde\Lambda_n^{(j)}=\frac1{2n}\eta_jR_{jj}+\frac1{2n}(R_{jj}-s_y(z))(1+\eta_{j0}).
\end{equation}

Note as well that random variables $X_{jl}$, for $l=1,\ldots,p$ and $\widetilde \Lambda_n^{(j)}$  are independent for all $j=1,\ldots,n$.
We continue with $H_1$ 
and  represent it in the form
\begin{align}\label{finish001}
 H_1:=H_{11}+H_{12}+H_{13},
\end{align}
where
\begin{align}
H_{11}&:=\frac1n\E\frac{\sum_{j=1}^n\varepsilon_{j\nu}
S_y(z)\varphi({\widetilde{\Lambda}_n^{(j)}})}{b_n^{(j)}(z)},\notag\\ 
H_{12}&:=\frac1n\E\frac{\sum_{j=1}^n\varepsilon_{j\nu}
S_y(z)(\varphi(\Lambda_n)-\varphi(\widetilde{\Lambda}_n^{(j)}))}{b_n^{(j)}(z)},\notag\\
H_{13}&:=-\frac1n\E\frac{\sum_{j=1}^n\varepsilon_{j\nu}
S_y(z)\widetilde\varepsilon_{j3}\varphi(\Lambda_n)}{b_n(z)b_n^{(j)}(z)},\notag
\end{align}
where

\begin{equation}
\widetilde\varepsilon_{j3}=m_n(z)-m_n^{(j)}(z)=\varepsilon_{j3}-\frac1{nz}.
\end{equation}
It is straightforward to check that, by conditional independence of $\varepsilon_{j\nu}$ and $\widetilde{\Lambda}_n^{(j)}$,
\begin{equation}\label{finish0001}
 H_{11}=0.
\end{equation}
Applying the Cauchy -- Schwartz inequality and  H\"older's inequality, and that $|\widetilde\varepsilon_{j3}|\le \frac C{n}$ for $V=4\sqrt y$, we get
\begin{align}\label{finish00001}
 |H_{13}|&\le \frac{C|S_y(z)|^{2}}{n}\E^{\frac1q}\Big(\frac1n\sum_{j=1}^n|\varepsilon_{j\nu}|^2\Big)^{\frac q2}\E^{\frac{q-1}q}|\varphi(\Lambda_n)|^{\frac q{q-1}}.
\end{align}
This inequality and Lemmas \ref{eps2} and \ref{eps3} in the Appendix, together imply
\begin{equation}
|H_{13}|\le \frac{Cq|S_y(z)|^{2}}{n}\E^{\frac{q-1}q}|\varphi(\Lambda_n)|^{\frac q{q-1}}
\end{equation}

We use  that 
\begin{equation}\notag
 \varepsilon_{j3}=\frac1{2p}(\Tr\mathbf R-\Tr\mathbf R^{(j)}+\frac1z)=\frac y{2n}(1+\eta_{j0}+\eta_j)R_{jj}+\frac y{2nz}.
\end{equation}
Note that 
\begin{equation}\label{lambda9}
 \delta_{nj}:=\Lambda_n-\widetilde{\Lambda}_n^{(j)}=\widetilde\varepsilon_{j3}-\frac {yS_y(z)(1+\eta_{j0})}{2n}-\frac{y}{2nz}=\frac y{2n}(R_{jj}-S_y(z))(1+\eta_{j0})+\frac1n\eta_jR_{jj}.
\end{equation}
This and that $|\eta_{j0}|\le V^{-2}$ yield
\begin{align}\label{gg1}
 |\delta_{nj}|\le \frac Cn|R_{jj}-S_y(z)|+\frac Cn|\eta_j|(|S_y(z)|+|R_{jj}-S_y(z)|).
\end{align}
By Taylor's formula $\varphi(x)-\varphi(y)=(x-y)\E_{\tau}\varphi'(x-\tau(x-y))$, we may write
\begin{align}\notag
 H_{12}=\frac{S_y(z)}n\E\frac{\sum_{j=1}^n\varepsilon_{j\nu}\delta_{nj}\varphi'(\Lambda_n-\tau\delta_{nj})
}{b_n^{(j)}(z)},
\end{align}
where $\tau$ denotes a uniformly distributed random variable on the unit interval which is independent of all other random variables.
It is straightforward to check that
\begin{equation}\label{lala}
 |\varphi'(\Lambda_n-\tau\delta_{nj})|\le q|\Lambda_n-\tau\delta_{nj}|^{q-2}\le Cq|\Lambda_n|^{q-2}+q^{q-1}|\tau\delta_{nj}|^{q-2}.
\end{equation}
Therefore, applying  the Cauchy -- Schwartz inequality, H\"older's inequality and finally inequality \eqref{lala}, we get
\begin{align}\notag
 |H_{12}|&\le C{q|S_y(z)|^2}\E^{\frac1{q}}(\frac1n\sum_{j=1}^n|\varepsilon_{j\nu}|^2)^{\frac q2}\E^{\frac1{q}}\Big(\frac1n\sum_{j=1}^n|\delta_{nj}|^2\Big)^{\frac q2}
 \E^{\frac{q-2}q}|\Lambda_n|^q\notag\\&\qquad\qquad\qquad\qquad+q^{q-1}|S_y(z)|^2\frac1n\sum_{j=1}^n\E|\varepsilon_{j\nu}||\delta_{nj}|^{q-1}.
\end{align}
Applying inequality \eqref{gg1} , we get
\begin{align}\notag
 |H_{12}|&\le K_1+K_2+K_3+K_4,
 \end{align}
 where
 \begin{align}
 K_1&:=\frac{Cq|S_y(z)|^2}n\E^{\frac1{q}}(\frac1n\sum_{j=1}^n|\varepsilon_{j\nu}|^2)^{\frac q2}\E^{\frac1{q}}(\frac1n\sum_{j=1}^n|R_{jj}-s(z)|^2)^{\frac q2}
 \E^{\frac{q-2}q}|\Lambda_n|^q,\notag\\
 K_2&:=\frac{Cq|S_y(z)|^3}n\E^{\frac1{q}}(\frac1n\sum_{j=1}^n|\varepsilon_{j\nu}|^2)^{\frac q2}\E^{\frac1{q}}
 (\frac1n\sum_{j=1}^n|\eta_{j}|^2)^{\frac q2}\E^{\frac{q-2}q}|\Lambda_n|^q,\notag\\
 K_3&:=\frac{Cq|S_y(z)|^2}n\E^{\frac1{q}}(\frac1n\sum_{j=1}^n|\varepsilon_{j\nu}|^2)^{\frac q2}\E^{\frac1{2q}}
 (\frac1n\sum_{j=1}^n|\eta_{j}|^4)^{\frac q2}\notag\\&\qquad\qquad\qquad\qquad\times
 \E^{\frac1{2q}}(\frac1n\sum_{j=1}^n|R_{jj}-s(z)|^4)^{\frac q2}\E^{\frac{q-2}q}|\Lambda_n|^q,\notag\\
K_4&:=\frac{q^{q-1}|S_y(z)|^2}n\sum_{j=1}^n\E|\varepsilon_{j\nu}||\delta_{nj}|^{q-1}.\notag
\end{align}
Using equality \eqref{lambda''} and $|\Lambda_n|\le \frac1{2\sqrt y}$ a. s., we get, for $z=u+iV$,
\begin{align}\label{lambdanew}
 |R_{jj}(z)-s(z)|&\le C|S_y(z)|^2(|\varepsilon_j|+|\varepsilon_j|^2+|\Lambda_n|+|\Lambda_n|^2)\notag\\&\le C|S_y(z)|^2(|\varepsilon_j|+|\varepsilon_j|^2+|\Lambda_n|).
\end{align}
Therefore,
\begin{align}
 \E^{\frac1{q}}\Big(\frac1n\sum_{j=1}^n|R_{jj}-S_y(z)|^2\Big)^{\frac q2}&\le C|S_y(z)|^4\Big(\E^{\frac1q}\Big(\frac1n\sum_{j=1}^n|\varepsilon_{j}|^2\Big)^{\frac q2}\notag\\&+
 \E^{\frac1q}\Big(\frac1n\sum_{j=1}^n|\varepsilon_{j}|^4\Big)^{\frac q2}+\E^{\frac1q}|\Lambda_n|^q\Big).
\end{align}
Applying the last inequality, we get
\begin{align}
 |K_1|&\le \frac{Cq|S_y(z)|^4}n\E^{\frac1{q}}(\frac1n\sum_{j=1}^n|\varepsilon_{j\nu}|^2)^{\frac q2}\E^{\frac1{q}}(\frac1n\sum_{j=1}^n|\varepsilon_{j}|^2)^{\frac q2}\E^{\frac{q-2}p}|\Lambda_n|^p\notag\\&+
 \frac{Cq|S_y(z)|^4}n\E^{\frac1{q}}(\frac1n\sum_{j=1}^n|\varepsilon_{j\nu}|^2)^{\frac q2}
 \E^{\frac1{q}}(\frac1n\sum_{j=1}^n|\varepsilon_{j}|^4)^{\frac q2}\E^{\frac{q-2}q}|\Lambda_n|^q\notag\\&+
 \frac{Cq|s(z)|^4}n\E^{\frac1{q}}(\frac1n\sum_{j=1}^n|\varepsilon_{j\nu}|^2)^{\frac q2}\E^{\frac{q-1}q}|\Lambda_n|^q.\notag
\end{align}
Using Corollary \ref{eps1} and Lemmas \ref{eps2}, \ref{eps3}, in the Appendix, we get
\begin{align}
 |K_1|&\le\frac{Cq^2|S_y(z)|^4}{n^2}
 \E^{\frac{p-2}p}|\Lambda_n|^p+\frac{Cp^2|S_y(z)|^2}{n}\E^{\frac{p-1}p}|\Lambda_n|^p.\notag
\end{align}
According to Corollary \ref{eps1} and Lemmas \ref{eps2}, \ref{eps3}, inequality \eqref{etaj}, in the Appendix we have
\begin{equation}\notag
 K_2\le \frac{Cq|S_y(z)|^4}{n^2}\E^{\frac{q-2}q}|\Lambda_n|^q,
\end{equation}
and
\begin{align}\notag
  K_3&\le\frac{Cq|S_y(z)|^4}{n^2}\E^{\frac{q-2}q}|\Lambda_n|^q.
\end{align}
To bound $K_4$ we use inequalities \eqref{gg1} and \eqref{lambdanew} and obtain
\begin{align}\label{k4}
 K_4&\le \frac{q^{q-1}|S_y(z)|^2}n\sum_{j=1}^n\E|\varepsilon_{j\nu}|
 \Big(\frac1n|R_{jj}-s(z)|(1+|\eta_j|)+\frac1n|\eta_j||S_y(z)|\Big)^{q-1}.
\end{align}

We rewrite now inequality \eqref{k4} in the form
\begin{align}
  K_4&\le \frac{q^{q-1}|S_y(z)|^2 }{n^{q-1}}\Big(\frac1n\sum_{j=1}^n\E|\varepsilon_{j\nu}|^q\Big)^{\frac1q}\Bigg(\frac1n\sum_{j=1}^n
 \E\Big(|R_{jj}-S_y(z)|(1+|\eta_j|)\notag\\&\qquad\qquad\qquad\qquad\qquad\qquad\qquad\qquad\qquad+|\eta_j||S_y(z)|\Big)^{q}\Bigg)^{\frac{q-1}q}.
\end{align}
Using again inequality \eqref{lambdanew},
we get
\begin{align}\label{k42}
K_4&\le \frac{C^qq^{q-1}|S_y(z)|^{2q} }{n^{q-1}}\Big(\frac1n\sum_{j=1}^n\E|\varepsilon_{j\nu}|^q\Big)^{\frac1q}
\Bigg(\frac1n\sum_{j=1}^n
 \E\Big((|\varepsilon_j|+|\Lambda_n|+|\varepsilon_j|^2)(1+|\eta_j|)^q\Bigg)^{\frac{q-1}q}\notag\\&+
\frac{C^qq^{q-1}|S_y(z)|^{q+1} }{n^{q-1}}\Big(\frac1n\sum_{j=1}^n\E|\varepsilon_{j\nu}|^q\Big)^{\frac1q} 
\Bigg(\frac1n\sum_{j=1}^n\E|\eta_j|^q\Bigg)^{\frac{q-1}q}.
\end{align}
Inequality \eqref{k42} and Lemmas \ref{eps1}, \ref{eps2} and \ref{eta} together  imply
\begin{equation}\label{k44}
 K_4\le \frac{C|S_y(z)|^{\frac{3q}2}q^{2q}}{n^{\frac{3q}2-1}}.
\end{equation}
Collecting the relations \eqref{finish1}, \eqref{finish01}, \eqref{finish001}, \eqref{finish0001}, \eqref{finish00001},  and \eqref{k44}
we get
\begin{align}
 \E|\Lambda_n|^q\le \frac{Cq|S_y(z)|^4}{n^2}\E^{\frac{q-2}q}|\Lambda_n|^q+\frac{C|S_y(z)|^{2}}{n}\E^{\frac{q-1}q}|\varphi(\Lambda_n)|^{\frac q{q-1}}
 +\frac{C|S_y(z)|^{\frac{3q}2}q^q}{n^q}.
\end{align}
Solving this inequality with respect to $\E|\Lambda_n|^q$, we get,
\begin{align}\label{disp}
 \E^{\frac1q}|\Lambda_n|^q\le\frac{Cq|S_y(z)|^{\frac32}}{n}.
\end{align}
\end{proof}
We return now to the proof of Theorem \ref{stlambda1}.  Applying Chebyshev's inequality and Proposition \ref{lambdalarge}, we get
\begin{equation}
\Pr\{|\Lambda_n|\ge \frac{Cq^{2}|S_y(z)|^{\frac{3}2}}{n}\}\le \frac{n^q\E|\Lambda_n|^q}{q^{2q}|S_y)z)|^{\frac{3q}2}}\le \frac 1{q^q}=\exp\{-c\log n\log\log n\}
\end{equation}
Thus Theorem \ref{stlambda1} is proved.
\section{The proof of Theorem \ref{stlambda}}Analogously to the proof of Theorem \ref{stlambda} we prove the following result.
\begin{prop}\label{lambdasmal}Under conditions of Theorem \ref{main} there exist constant $A$ and $C$ such that for any $q\le A\log n$ and for all $z\in\mathbb G$,
\begin{equation}
\E|\Lambda_n|^q\le \frac{C^qq^q\lna^q}{(nv)^q}.
\end{equation}
\end{prop}
\begin{proof}The proof of Proposition \ref{lambdasmal} is similar to the proof of Proposition \ref{lambdalarge}. We consider the equality
\begin{equation}
\E|\Lambda_n|^q=\E\Lambda_n\varphi(\Lambda_n)=\E\frac{T_n}{b_n(z)}\varphi(\Lambda_n),
\end{equation}
where
\begin{equation}
T_n=T_{n1}+T_{n2}+T_{n3},
\end{equation}
and 
\begin{equation}
T_{n\nu}=\frac1n\sum_{j=1}^n\varepsilon_{j\nu}R_{jj}, \quad\nu=1,2,3.
\end{equation}
We start from $T_{n3}$.  Using representation \eqref{e3} in the Appendix, it is straightforward to check that
\begin{equation}
T_{n3}=-\frac y{2n}m_n'(z)+\frac y{2nz}.
\end{equation}
This equality implies that
\begin{equation}
|T_{n3}|\le \frac 1{nv}\im m_n(z)+\frac1{n|z|}.
\end{equation}
From the other hand side
\begin{equation}
|b_n(z)|\ge y\im m_n(z)+(1-y)\im \frac1z=y\im m_n(z)+(1-y) \frac v{|z|^2}.
\end{equation}
This implies that, for $z\in\mathbb G$,
\begin{equation}
|T_{n3}||b_n^{-1}(z)|\le \frac C{nv}.
\end{equation}
Applying this ineqaulity and Jensen inequality we get
\begin{equation}\label{546}
\E|\frac{T_{n3}}{b_n(z)}\varphi(\Lambda_n)|\le \frac C{nv}\E^{\frac{q-1}q}|\Lambda_n(z)|^q.
\end{equation}
Now we consider $\nu=1,2$. We write the representation for $\Gamma_{\nu}:=\frac1n\E\frac{\sum_{j=1}^n\varepsilon_{j\nu}
 R_{jj}\varphi(\Lambda_n)}{b_n(z)}$ similar to representation \ref{nu12}
 \begin{align}
\Gamma_{\nu}=H_1+H_2,\notag
\end{align}
where
\begin{align}
 H_1&:=-\frac1n\E\frac{\sum_{j=1}^n\varepsilon_{j\nu}
\frac1{a_n^{(j)}(z)}\varphi(\Lambda_n)}{b_n(z)},\notag\\
H_2&:=\frac1n\E\frac{\sum_{j=1}^n\varepsilon_{j\nu}(R_{jj}+\frac1{a_n^{(j)}(z)}
)\varphi(\Lambda_n)}{b_n(z)}.
\end{align}
Using equality \eqref{repr01}, we may write
\begin{equation}
H_2:=\frac1n\sum_{j=1}^n\E\frac{\varepsilon_{j\nu}(\varepsilon_{j1}+\varepsilon_{j2})R_{jj}\varphi(\Lambda_n)}{a_n(z)b_n(z)}
\end{equation}
By inequality $ab\le \frac12(a^2+b^2)$, we have, for some numerical constant $C$,
\begin{equation}
|H_2|\le \frac Cn\sum_{\nu=1}^2\sum_{j=1}^n\E\frac{|\varepsilon_{j\nu}|^2|R_{jj}||\varphi(\Lambda_n)|}{|a_n^{(j)}(z)b_n(z)|}
\end{equation}
Applying H\"older's inequality, we get
\begin{equation}
|H_2|\le \frac Cn\sum_{\nu=1}^2\sum_{j=1}^n\E^{\frac1q}\frac{|\varepsilon_{j\nu}|^{2q}|R_{jj}|^{q}}{|a_n^{(j)}(z)b_n(z)|^q}
\E^{\frac{q-1}q}|\varphi(\Lambda_n)|^{\frac q{q-1}}.
\end{equation}
We estimate now the factor $D_{n\nu}^{q}:=\E\frac{|\varepsilon_{j\nu}|^{2q}|R_{jj}|^{q}}{|a_n^{(j)}(z)b_n(z)|^q}$.
Applying H\"older's inequality, we get
\begin{align}
D_{n\nu}^{q}\le \E^{\frac12}|R_{jj}|^{2q}
\E^{\frac12}\frac{|\varepsilon_{j\nu}|^{4q}}{|b_n(z)|^{2q}|a_n^{(j)}(z)|^{2q}}.
\end{align}
Using Lemma \ref{nana}, we get, for $z\in\mathbb G$,
\begin{equation}
D_{n\nu}^{q}\le C\E^{\frac12}|R_{jj}|^{2q}\E^{\frac12}\frac{|\varepsilon_{j\nu}|^{4q}}{|b_n^{(j)}(z)|^{2q}|a_n^{(j)}(z)|^{2q}}
\end{equation}
Conditioning on $\mathfrak M^{(j)}$ and applying Lemmas \ref{eps1} and \ref{eps2} in the Appendix, we obtain
\begin{align}
D_{n\nu}^{q}\le\frac{ C^qq^{2q}\lna^q}{(nv)^q}\E^{\frac12}|R_{jj}|^{2q}\E^{\frac12}\frac{(\im m_n^{(j)}(z))^{2q}}
{|b_n^{(j)}(z)|^{2q}|a_n^{(j)}(z)|^{2q}}\le \frac{ C^qq^{2q}\lna^q}{(nv)^q}\E^{\frac12}|R_{jj}|^{2q}\E^{\frac12}\frac{1}
{|a_n^{(j)}(z)|^{2q}}
\end{align}
The last inequality and Propositions \ref{rjj8} and \ref{an} in Section \ref{diag} together imply
\begin{align}
D_{n\nu}^{q}\le \frac{ C^qq^{2q}\lna^q}{(nv)^q},
\end{align}
and
\begin{equation}\label{555}
H_2\le \frac{ Cq^{2}\lna}{nv}\E^{\frac{q-1}q}|\varphi(\Lambda_n)|^{\frac q{q-1}}.
\end{equation}
We shall estimate now the quantity $H_1$. We represent it in the form
\begin{equation}
H_1=H_{11}+H_{12}+H_{13},
\end{equation}
where
\begin{align}
H_{11}&=-\frac1n\sum_{j=1}^n\E\frac{\varepsilon_{j\nu}\varphi({\widetilde\Lambda}_n^{(j)}(z))}{a_n^{(j)}b_n^{(j)}},\notag\\
H_{12}&=-\frac1n\sum_{j=1}^n\E\frac{\varepsilon_{j\nu}(b_n(z)-b_n^{(j)}(z))\varphi({\Lambda}_n(z))}{a_n^{(j)}b_n^{(j)}b_n(z)},\notag\\
H_{13}&=-\frac1n\sum_{j=1}^n\E\frac{\varepsilon_{j\nu}(\varphi({\Lambda}_n(z)-{\widetilde\Lambda}_n^{(j)}(z))}{a_n^{(j)}b_n^{(j)}}.
\end{align}
Since $\varepsilon_{j\nu}$ and ${\widetilde\Lambda}_n^{(j)}(z))$, $ a_n^{(j)}$, and $b_n^{(j)}(z)$ are independent, we have
\begin{equation}\label{558}
H_{11}=0.
\end{equation}
To estimate $H_{12}$ we apply H\"older's inequality.
We get
\begin{equation}
|H_{12}|\le \frac1n\sum_{j=1}^n\E^{\frac1{2q}}\frac{|\varepsilon_{j\nu}|^{2q}}{|b_n^{(j)}|^{2q}|a_n^{(j)}|^{2q}}
\E^{\frac1{2q}}\frac{|b_n(z)-b_n^{(j)}(z)|^{2q}}{|b_n(z)|^{2q}}
\E^{\frac{q-1}q}|\varphi(\Lambda_n)|^{\frac q{q-1}}.
\end{equation}
Using Lemmas \ref{bn}, \ref{eps1} and \ref{eps2}, we get
\begin{equation}\label{560}
|H_{12}|\le \frac{C^q}{(nv)^q}\E^{\frac{q-1}q}|\varphi(\Lambda_n)|^{\frac q{q-1}}\le \frac{C^q}{(nv)^q}\E^{\frac{q-1}q}|\Lambda_n|^q.
\end{equation}
We estimate now $H_{13}$. 
Note that, by Taylor's formula
\begin{equation}
\varphi({\Lambda}_n(z))-\varphi({\widetilde\Lambda}_n^{(j)}(z))=({\Lambda}_n(z)-{\widetilde\Lambda}_n^{(j)}(z))\E_{\tau}
\varphi'({\Lambda}_n(z)+\tau({\widetilde\Lambda}_n^{(j)}(z)-{\Lambda}_n(z)))
\end{equation}
where $\tau$ is uniform distributed on the unit interval random variable independent of all other r.v.'s.
Note that
\begin{equation}
|\varphi'({\Lambda}_n(z)+\tau({\widetilde\Lambda}_n^{(j)}(z)-{\Lambda}_n(z)))|\le Cq|\Lambda_n|^{q-2}
+C^qq^q|{\widetilde\Lambda}_n^{(j)}(z)-{\Lambda}_n(z))|^{q-2}.
\end{equation}
We recall equality \eqref{lambda9},
\begin{equation}\label{lambda9*}
 \Lambda_n(z)-{\widetilde\Lambda}_n^{(j)}(z))=\delta_{nj}=\frac y{2n}(R_{jj}-S_y(z))(1+\eta_{j0})+\frac yn\eta_jR_{jj}.
\end{equation}
By equality \eqref{lambda''}, we have
\begin{equation}\label{lamb}
 R_{jj}-S_y(z)=-S_y(z)\varepsilon_jR_{jj}-S_y(z)R_{jj}\Lambda_n(z).
\end{equation}
Combining these relations, we get
\begin{equation}
   \Lambda_n(z)-{\widetilde\Lambda}_n^{(j)}(z)) =   -\frac{yS_y(z)}{2n} \varepsilon_jR_{jj}(1+\eta_{j0})-\frac{yS_y(z)}{2n}  \Lambda_n(z)R_{jj} (1+\eta_{j0}) +\frac yn\eta_jR_{jj}.   
\end{equation}
Using this representation, we may write
\begin{align}
 |H_{13}|\le \widetilde H_1+\cdots+\widetilde H_3,
\end{align}
where
\begin{align}
 \widetilde H_1&:=\frac{Cq}{n^2}\sum_{j=1}^n\E\frac{|\varepsilon_{j\nu}|^2(1+|\eta_{j0}|)}{|b_n^{(j)}(z)||a_n^{(j)}(z)|}|\Lambda_n|^{q-2},\notag\\
 \widetilde H_2&:=\frac{Cq}{n^2}\sum_{j=1}^n\E\frac{|\varepsilon_{j\nu}|(1+|\eta_{j0}|)}{|b_n^{(j)}(z)||a_n^{(j)}(z)|}|\Lambda_n|^{q-1},\notag\\
 \widetilde H_3&:=\frac{Cq}{n^2}\sum_{j=1}^n\E\frac{|\varepsilon_{j\nu}||\eta_j|}{|b_n^{(j)}(z)||a_n^{(j)}(z)|}|\Lambda_n|^{q-2},\notag\\
 \widetilde H_4&=\frac{C^qq^q}{n^q}\sum_{j=1}^n\E\frac{|\varepsilon_{j\nu}||\varepsilon_j|^{q-1}(1+|\eta_{j0}|)^{q-1}|}{|b_n^{(j)}(z)||a_n^{(j)}(z)|}|R_{jj}|^{q-1},\notag\\
 \widetilde H_5&=\frac{C^qq^q}{n^q}\sum_{j=1}^n\E\frac{|\varepsilon_{j\nu}|(1+|\eta_{j0}|)^{q-1}|}{|b_n^{(j)}(z)||a_n^{(j)}(z)|}|\Lambda_n|^{q-1}|R_{jj}|^{q-1},\notag\\
 \widetilde H_6&=\frac{C^qq^q}{n^q}\sum_{j=1}^n\E\frac{|\varepsilon_{j\nu}|\eta_j|^{q-1}}{|b_n^{(j)}(z)||a_n^{(j)}(z)|}|R_{jj}|^{q-1}.
\end{align}

Note that
\begin{equation}
 1+|\eta_{j0}|\le v^{-1}\im(z+\im m_n^{(j)}(z))\le v^{-1}|b_n^{(j)}(z)|.
\end{equation}
Using this inequality and applying H\"older's inequality it is straightforward to check that
\begin{equation}
 \max\{\widetilde H_1, \widetilde H_3\}\le \frac{Cq^4\lna^{2q}}{(nv)^2}\E^{\frac{q-2}q},|\Lambda_n|^q
\end{equation}
and
\begin{equation}
 \widetilde H_2\le \frac {Cq\lna^q}{nv}\E^{\frac{q-1}q}|\Lambda_n|^q.
\end{equation}
Applying H\"older's inequality and lemmas \ref{eps1},\ref{eps2}, \ref{eta_j},  we prove that
\begin{equation}
 \max\{\widetilde H_4, \widetilde H_6\}\le \frac{C^qq^q\lna^q}{(nv)^q}.
\end{equation}
Finally, using that, for $z\in\mathbb G$,
\begin{equation}
 |\Lambda_n|\le c|T_n|^{\frac12},
\end{equation}
we get
\begin{equation}
 \widetilde H_5\le \frac{C^qq^q\lna^q}{(nv)^q}.
\end{equation}
Combining these inequalities, we get
\begin{align}\label{566}
 |H_{13}|\le \frac{Cq\lna^2}{(nv)^2}\E^{\frac{q-2}q}|\Lambda_n|^q+\frac {Cq\lna^q}{nv}\E^{\frac{q-1}q}|\Lambda_n|^q+\frac{C^qq^q}{(nv)^q}.
\end{align}

 Combining inequalities \eqref{546}, \eqref{555}, \eqref{560}, \eqref{566} and equality \eqref{558},  we obtain
 \begin{equation}
 \E|\Lambda_n(z)|^q\le \frac {Cq\lna^2}{(nv)^2}\E^{\frac{q-2}q}|\Lambda_n(z)|^q+\frac{Cq\lna}{nv}\E^{\frac{q-1}q}|\Lambda_n(z)|^q+\frac{C^qq^q\lna^q}{(nv)^q}.
 \end{equation}
Solving this inequality with respect to $\E|\Lambda_n(z)|^q$, we get
\begin{equation}
 \E|\Lambda_n(z)|^q\le \frac{C^qq^q\lna^q}{(nv)^q}=\frac{C^q\beta_n^q}{(nv)^q}
\end{equation}
Thus Proposition \ref{lambdasmal}  is proved.
\end{proof}
We return now to the proof of  Theorem \ref{stlambda}. Applying Chebyshev inequality and result of Proposition \ref{lambdasmal} with $q=A\log n$, we arrive
\begin{equation}
 \Pr\{|\Lambda_n|\ge \frac {\beta_n^2}{nv}\}\le \frac{\E|\Lambda_n|^q(nv)^q}{\beta_n^{2q}}\le \frac{C^q}{\beta_n^q}\le \exp\{-c\lna\}.
\end{equation}
The last inequality completes the proof of Theorem \ref{stlambda}









\section{Proof of Theorem \ref{eigenvector}}
Consider the singular value decomposition of the matrix $\mathbf X$. 
Let $\mathbf U$ and $\mathbf H$ be
unitary matrices of dimension $n\times n$ and $p\times p$ respectively. 
Let $\mathbf S$ be a $n\times n$ diagonal matrix whose entries are the   
singular value of the matrix $\mathbf X$. Let $\mathbf O_{p\times q}$ denote the  
$p\times q$-matrix  with zero entries. Introduce the matrix
$\widetilde{\mathbf S}=\begin{bmatrix}\mathbf S\quad\mathbf O_{n\times(p-n)}\end{bmatrix}$.
We have the following representation
\begin{equation}\label{SVD}
 \mathbf X=\mathbf U\widetilde{\mathbf S}\mathbf H^*.
\end{equation}
We may  represent the matrix $\mathbf H$ in the form
\begin{equation}
 \mathbf H=\begin{bmatrix}&\mathbf H_{11}&\mathbf H_{12}\\&\mathbf H_{21}
&\mathbf H_{22}\end{bmatrix},
\end{equation}
where $\mathbf H_{11}$ is $n\times n$ matrix , $\mathbf H_{22}$ is $(p-n)\times (p-n)$ matrix.
We introduce matrix
\begin{equation}
 \mathbf Z^*=\begin{bmatrix}&\frac1{\sqrt 2}\mathbf U^*&\frac1{\sqrt 2}
\mathbf H_{11}^*&\frac1{\sqrt 2}\mathbf H_{12}^*
\\&\frac1{\sqrt 2}\mathbf U^*
&-\frac1{\sqrt 2}\mathbf H_{11}^*&-\frac1{\sqrt 2}\mathbf H_{12}^*\\
            &\mathbf O&\mathbf H_{21}^*&\mathbf H_{22}^*\end{bmatrix}.
\end{equation}
It is straightforward to check that
\begin{equation}\label{eigenvektor10}
 \mathbf Z^*\mathbf V\mathbf Z=
\begin{bmatrix}&\mathbf S&\mathbf O_{n\times n}&\mathbf O_{n\times (p-n)}\\
                                &\mathbf O_{n\times n}&-\mathbf S&\mathbf O_{n\times (p-n)}\\
&\mathbf O_{(p-n)\times n}&\mathbf O_{(p-n)\times n}&\mathbf O_{(p-n)\times(p- n)}
                               \end{bmatrix},
\end{equation}
where $\mathbf S$ denotes diagonal matrix with  entries $s_j$.
The equality (\ref{eigenvektor10}) implies that the rows $\mathbf z_{j}$ of the matrix 
$\mathbf Z$, for $j=1,\ldots,n$,  are  the
eigenvectors of the matrix $\mathbf V$ corresponding to the eigenvalues $s_j$.
 Similarly, the  rows $\mathbf z_{j+n}$ of the  matrix $\mathbf Z$, for $j=1,\ldots,n$, 
are  the
eigenvectors of the  matrix $\mathbf V$ corresponding to the eigenvalues $-s_j$ and the 
rows $\mathbf z_{2n+l}$, for $l=1,\ldots,p-n$, are
the eigenvectors of matrix $\mathbf V$ corresponding to the eigenvalues $0$.
rjj8
We note the following representation for the diagonal entries of the resolvent matrix 
$\mathbf R$:
\begin{equation}\label{representeigenvector}
 R_{jj}=\sum_{k=1}^{n+p}\frac1{\lambda_k-z}|Z_{kj}|^2.
\end{equation}
Denote by $\lambda_1,\ldots,\lambda_{n+p}$  the eigenvalues of the  matrix $\mathbf V$ 
ordered in such way that
\begin{equation}
\lambda_j=\begin{cases}-s_j,\quad\text{if}\quad 1\le j\le n\\s_j,\quad\text{if}
\quad n+1\le j\le 2n\\0,\quad\text{if}\quad 2n\le j\le
n+p.\end{cases}
\end{equation}

Consider the  distribution function $F_{nj}(x)$ of the following  weighted empirical 
probability distribution on the eigenvalues  $\lambda_1, \ldots, \lambda_{n+p}$
\begin{equation}
 F_{nj}(x)=\sum_{k=1}^{n+p}|Z_{kj}|^2\mathbb I\{\lambda_k\le x\}.
\end{equation}
Then we have
\begin{equation}
 R_{jj}= R_{jj}(z)=\int_{-\infty}^{\infty}\frac1{x-z}dF_{nj}(x).
\end{equation}
which means  that $R_{jj}$ is the Stieltjes transform of the distribution $F_{nj}(x)$.
Note that, for any $\lambda>0$
\begin{equation}
 \max_{1\le k\le n+p}|Z_{kj}|^2\le \sup_x(F_{nj}(x+\lambda)-F_{nj}(x))=:Q_{nj}(\lambda).
\end{equation}
On the  other hand, it is easy to check that
\begin{equation}\label{concentration}
 Q_{nj}(\lambda)\le 2\sup_u\lambda\im R_{jj}(u+i\lambda).
\end{equation}
Furthermore, Corollary \ref{cru3} and inequality \ref{u1} together imply
together imply, for $v\ge v_0$
\begin{equation}\label{aaa}
 \Pr\{\im R_{jj}\le |R_{jj}|\le C\}\le\exp\{-c\lna\}.
\end{equation}
This implies that
\begin{equation}
 \Pr\{\max_{1\le k\le n+p}|Z_{kj}|^2>\frac{C\beta_n^4}{n}\}\le C\exp\{-c\4 \lna\}.
\end{equation}
By definition of $\mathbf H$, we obtain
\begin{equation}
 \Pr\{\max_{1\le j,k\le n}|u_{kj}|^2>\frac{C\beta_n^4}{n}\}\le C\exp\{-c\4 \lna\}.
\end{equation}
and
\begin{equation}
 \Pr\{\max_{1\le j,k\le p}|v_{kj}|^2> \frac{C\beta_n^4}{n}\}\le C\exp\{-c\4 \lna\}.
\end{equation}
By a  union bound, the inequality \eqref{deloc} follows.
To prove inequality \eqref{deloc1}, we consider the quantity
\begin{equation}
 r_j:=R_{jj}-S_y(z),\quad j=1,\ldots,n.
\end{equation}
By equality \eqref{lamb}, we have
\begin{equation}
 r_j=-S_y(z)(\Lambda_n(z)+\varepsilon_j)R_{jj}.
\end{equation}
By  Proposition \ref{prop10} and Lemma \ref{epsilon} , we have
\begin{equation}
\Pr\{ |r_j|\le \frac{C\beta_n^2}{\sqrt{nv}}\}\ge 1-\exp\{-c\lna^2\}.
\end{equation}
This implies that
\begin{equation}
 \Pr\{\int_{v'}^V|r_j(x+iv)|dv\le \frac {C\beta_n^2}{\sqrt n}
\}\ge 1-\exp\{-c\lna^2\}.
\end{equation}
Similar to \eqref{daleko} we get
\begin{equation}
 \int_{-\infty}^{\infty}|r_j(x+iV)|dx\le \frac {C\beta_n^2}{\sqrt n}.
\end{equation}
Applying Corollary \ref{smoothing1}, we finally obtain
\begin{equation}
 \Pr\{\sup_x|F_{nj}(x)-G_y(x)|\le \frac{C\beta_n^2}{\sqrt n}\}\ge 1-C\exp\{-c \4 \lna^2\}.
\end{equation}
In view of
\begin{equation}
\Pr\{\sup_x|F_{n}(x)-G_y(x)|\le \frac{C\beta_n^4}{ n}\}\ge 1-C\exp\{-c\4 \lna^2\},
\end{equation}
we get
\begin{equation}
\Pr\{\sup_x|F_{nj}(x)-G_y(x)|\le \frac{C\beta_n^2}{\sqrt n}\}\ge 1-C\exp\{-c\4 \lna^2\}.
\end{equation}
The last two inequalities together imply that
\begin{equation}
 \Pr\{\sup_x|F_{nj}(x)-F_n(x)(x)|\le \frac{C\beta_n^2}{\sqrt n}\}\ge 1-C\exp\{-c\4 \lna^2\}.
\end{equation}
Note that $F_n(x)$ is the distribution function of a random variable which is uniformly  
distributed on the set  $\{\pm s_1,\ldots,\pm s_n\}$ and
\begin{equation}
 \sup_x|F_{nj}(x)-F_n(x)|=\max_k\Big|\sum_{l=1}^k|u_{lj}|^2-\frac kn\Big|.
\end{equation}
Thus Theorem \ref{eigenvector} is proved.
\section{Appendix}
\subsection{Proof of Remark \ref{localization}}\label{remark}
\begin{Proof of} {\it Remark \ref{localization}.}We consider the case $y<1$ only. For the case $y=1$ see \cite{GT:11}, Remark 1.2.
Let $\alpha=G_y(1+y)$. Let $x\in[0,\alpha]$.
Denote  by $\tau$  a random variable which is uniformly distributed on $[0,1]$ and use
  Taylor's formula to show
 \begin{equation}\label{tailor}
  G^{-1}(x)=a+ \E_{\tau}\frac{2\pi xG_y^{-1}(x\tau)}{\sqrt{(b-(G_y^{-1}(x\tau))(G_y^{-1}(x\tau)-a)}}.
 \end{equation}
 By monotonicity of $\sqrt{(b-(G_y^{-1}(x\tau))(G_y^{-1}(x\tau)-a)}$ for $x\in [0,\alpha]$ and $G^{-1}_y(x)\ge a$, we get
 \begin{equation}
 G_y^{-1}(x)-a \ge \frac{C xa}{\sqrt{(b-(G_y^{-1}(x))(G_y^{-1}(x)-a)}}.
 \end{equation}
 There is another
   constant $C>0$ depending on $y$ such that
 \begin{equation}
  G_y^{-1}(x)-a\ge C x^{\frac23}.
 \end{equation}
From the last inequality we get
\begin{equation} \label{onethird}
\sqrt{G_y^{-1}(x\tau)-a} \ge c (\tau x)^{1/3}
\end{equation}
and hence by  \eqref{tailor} it follows that
\begin{equation}
G_y^{-1}(x)-a \le  c'x^{\frac23}\E_{\tau}\frac1{\tau^{\frac13}}\le C_2x^{\frac23},
\end{equation}
with some  constants $c',C_2>0$ depending on $y$ only.
Similarly  for $x\in[\alpha,1]$ we get
\begin{equation}
 C_1(1-x)^{\frac23}\le b-G^{-1}_y(x)\le C_2(1-x)^{\frac23}.
\end{equation}
Summarizing, we may write for another  constant $C_1$, depending  $y$ only,
\begin{equation}\label{new1}
 C_1\min\{x^{\frac23},(1-x)^{\frac23}\}\le (b-(G_y^{-1}(x\tau))(G_y^{-1}(x\tau)-a)\le C_1\min\{x^{\frac23},(1-x)^{\frac23}\}.
\end{equation}
Furthermore,
 \begin{equation}
  \Delta_n^*=\sup_x|\mathcal F_n(x)-G_y(x)|=\max_{1\le j\le n}|\mathcal F_n(s_j^2)-G_y(s_j^2)|=
  \max_{1\le j\le n}|\frac jn-G_y(s_j^2)|.
 \end{equation}
This implies that, for $s_j^2\in[a,b]$ and $\abs{\theta} \le 1$,
\begin{equation}
 s_j^2=G^{-1}_y(\frac jn+\theta\Delta_n^*).
\end{equation}
By Taylor's formula we have
\begin{equation}
 G^{-1}_y(\frac jn+\theta\Delta_n^*)=G^{-1}_y(\frac jn)+\E_{\tau}\frac{2\pi\theta\Delta_n^*G^{-1}(\frac kn+\tau\theta\Delta_n^*)}
 {\sqrt{(b-(G^{-1}(\frac kn+\tau\theta\Delta_n^*))(G^{-1}(\frac kn+\tau\theta\Delta_n^*)-a)}}.
\end{equation}
Consider first the case $2\Delta_n^*\le \frac jn\le \alpha-\Delta_n^*$.
Then by  \eqref{new1},
\begin{equation}
 \sqrt{(b-G^{-1}_y(\frac kn+\tau\theta\Delta_n^*))(G^{-1}_y(\frac kn+\tau\theta\Delta_n^*)-a)}\ge C|\frac jn+\tau\theta\Delta_n^*|^{\frac13}\ge
 C'(\frac jn)^{\frac13}.
\end{equation}
From here it follows that by Theorem \ref{main},  with probability
$1-C \exp\{-c l_{n,\alpha}\}$,
\begin{equation}
 |s_j^2-\gamma_{nj}|\le C\beta_n^4n^{-\frac23}j^{-\frac13}.
\end{equation}
Similar we get, for $2\Delta_n^*\le \frac{n-j}n\le \alpha-\Delta_n^*$,
\begin{equation}
 |s_j^2-\gamma_{nj}|\le C\beta_n^4n^{-\frac23}(n-j)^{-\frac13}.
\end{equation}

Thus Remark \ref{localization} is proved.
  \end{Proof of}
  \subsection{The proof of Proposition \ref{smoothing}}\label{profofsmoothing}
  
  \subsubsection{Auxiliary lemmas}We prove first several lemmas about behavior of distribution function of Marchenko--Pastur law and its Stieltjes transform.
\begin{lem}\label{densitymp} Let $0<y<1$.
 Let $x:\4|x|\in[1-\sqrt y,1+\sqrt y]$ and let 
$\gamma:=\gamma(x)=\min\{|x|-1+\sqrt y,1+\sqrt y-|x|$.
Then, for $0<y<1$, 
\begin{equation}
 |G_y'(x)|\le \frac{3\gamma}{\pi\sqrt{y(1-\sqrt y)}}.
\end{equation}
\end{lem}
\begin{proof}
 By equality \eqref{march}, we have
\begin{equation}
G'_y(x)=\frac{\sqrt{(-1+\sqrt y)^2-x^2)((1+\sqrt y)-x^2)}}{2\pi y|x|}
\mathbb I\{1-\sqrt y\le |x|\le 1+\sqrt y\}.
\end{equation}
Assume for the definitely that $x=-1+\sqrt y-\gamma$. Note that $0\le\gamma\le \sqrt y<1$.
It is straightforward to check that
\begin{align}
 x-1+\sqrt y&=-2+2\sqrt y-\gamma,\quad x+1-\sqrt y=-\gamma,\notag\\
 1+\sqrt y-x&=2+\gamma,\quad 1+\sqrt y+x=2\sqrt y-\gamma.
\end{align}
We may write
\begin{equation}
G'_y(x)=\frac{\sqrt{2\gamma( 1-\sqrt y+\frac12\gamma)(2+\gamma)(2\sqrt y+\gamma)}}
{2\pi y(1-\sqrt y+\gamma)}
\le \frac{3\sqrt{\gamma}}{\pi\sqrt{y(1-\sqrt y)}}.
\end{equation}
Similar we consider the cases $x=-1-\sqrt y+\gamma$, $x=1-\sqrt y+\gamma$, 
and $x=1+\sqrt y-\gamma$.
Thus Lemma \ref{densitymp} is proved.
\end{proof}
\begin{lem}\label{supremum} For any distribution function $F$ and for any 
$\frac12\sqrt y>\varepsilon>0$ the following inequality holds
\begin{equation}
 \sup_x|F(x)-\widetilde G_y(x)|\le \sup_{x:\4|x|
\in[1-\sqrt y+\varepsilon,1+\sqrt y-\varepsilon]}|F(x)-\widetilde G_y(x)|+
\frac{2\varepsilon^{\frac32}}{\pi\sqrt{y(1-\sqrt y)}}.
\end{equation}

\end{lem}
\begin{proof}Recall that $\mathbb J_{\varepsilon}
=[1-\sqrt y+\varepsilon,1+\sqrt y-\varepsilon]$. Note that
\begin{align}
 \sup_x&|F(x)-\widetilde G_y(x)|=\sup_{x:\4|x|\in[1-\sqrt y,1+\sqrt y]}
|F(x)-\widetilde G_y(x)|\notag\\&=
\max\Bigg\{\sup_{x\in[-1-\sqrt y,-1-\sqrt y+\varepsilon]}|F(x)-\widetilde G_y(x)|,
\sup_{x\in[-1+\sqrt y-\varepsilon,-1+\sqrt y]}|F(x)-\widetilde G_y(x)|,
\notag\\&\qquad\qquad\quad\sup_{x\in[1-\sqrt y,1-\sqrt y+\varepsilon]}
|F(x)-\widetilde G_y(x)|,
\sup_{x\in[1+\sqrt y-\varepsilon,1+\sqrt y]}|F(x)-\widetilde G_y(x)|,
\notag\\&\qquad\qquad\qquad\qquad\sup_{x:\4|x|\in\mathcal J_{\varepsilon}}
|F(x)-\widetilde G_y(x)|\Bigg\}.
\end{align}
Futhermore, for $x\in[-1-\sqrt y,-1-\sqrt y+\varepsilon]$, we have
\begin{align}\label{ineq2.7}
 -\widetilde G_y(-1-\sqrt y+\varepsilon)&\le F(x)-\widetilde G_y(x)\notag\\&
\le F(-1-\sqrt y+\varepsilon)-\widetilde G_y(-1-\sqrt y+\varepsilon)
+\widetilde G_y(-1-\sqrt y+\varepsilon).
\end{align}
Inequality \eqref{ineq2.7} implies that
\begin{equation}\notag
\sup_{x\in[-1-\sqrt y,-1-\sqrt y+\varepsilon]}|F(x)-\widetilde G_y(x)|\le \sup_{|x|
\in\mathcal J_{\varepsilon}'}|F(x)-\widetilde G_y(x)|+ 
\widetilde G_y(-1-\sqrt y+\varepsilon).
\end{equation}
Similar we get
\begin{equation}\notag
 \sup_{x\in[1+\sqrt y-\varepsilon,1+\sqrt y]}|F(x)-\widetilde G_y(x)|
\le \sup_{|x|\in\mathcal J_{\varepsilon}}|F(x)-\widetilde G_y(x)|+ 
1-\widetilde G_y(1+\sqrt y-\varepsilon).
\end{equation}
Furthermore, for $x\in[-1+\sqrt y-\varepsilon,-1+\sqrt y]$ we have
\begin{align}\label{ineq2.8}
 F(-1+\sqrt y-\varepsilon)&-\widetilde G_y(-1+\sqrt y-\varepsilon)
-(\widetilde G_y(-1+\sqrt y)-\widetilde G_y(-1+\sqrt y-\varepsilon))
\notag\\&\le F(x)-\widetilde G_y(x)
\notag\\   \le F(1-\sqrt y+\varepsilon)&-\widetilde G_y(1-\sqrt y+\varepsilon)
+(\widetilde G_y(1-\sqrt y+\varepsilon)
-\widetilde G_y(1-\sqrt y)).
\end{align}
We use here that $\widetilde G(-1+\sqrt y)=1-\widetilde G_y(1-\sqrt y)$. 
Inequality \eqref{ineq2.8} implies 
\begin{equation}\notag
\sup_{x\in[-1-\sqrt y,-1-\sqrt y+\varepsilon]}|F(x)-\widetilde G_y(x)|
\le \sup_{|x|\in\mathcal J_{\varepsilon}}|F(x)-\widetilde G_y(x)|+ 
\widetilde G_y(1-\sqrt y+\varepsilon)-\widetilde G_y(1-\sqrt y). 
\end{equation}
Similar we get
\begin{equation}\notag
\sup_{x\in[1-\sqrt y,1-\sqrt y-\varepsilon]}|F(x)-\widetilde G_y(x)|
\le \sup_{|x|\in\mathcal J_{\varepsilon}'}|F(x)-\widetilde G_y(x)|+ 
 \widetilde G_y(1-\sqrt y+\varepsilon)-\widetilde G_y(1-\sqrt y).
\end{equation}
We use here that for $\widetilde G_y(x)$
\begin{equation}
\widetilde G_y(1-\sqrt y+\varepsilon)-\widetilde G_y(1-\sqrt y)= 
 \widetilde G_y(-1+\sqrt y)-\widetilde G_y(-1+\sqrt y-\varepsilon).                   
\end{equation}
These relations together imply
\begin{align}\notag
 \sup_x|F(x)-\widetilde G_y(x)|&\le\sup_{|x|\in\mathcal J_{\varepsilon}'}
|F(x)-\widetilde G_y(x)|\notag\\&
+ \max\{\widetilde G_y(-1-\sqrt y+\varepsilon),
\widetilde G_y(-1+\sqrt y)-\widetilde G_y(-1+\sqrt y-\varepsilon)\}.
\end{align}
We note as well that, by Lemma \ref{densitymp},
\begin{equation}
\max\{\widetilde G_y(-1-\sqrt y+\varepsilon),
\widetilde G_y(-1+\sqrt y)-\widetilde G_y(-1+\sqrt y-\varepsilon)\}
\le \frac{2\varepsilon^{\frac32}}{\pi\sqrt{y(1-\sqrt y)}}.
\end{equation}
 
\end{proof}

  \subsubsection{The Proof of Proposition}
\begin{Proof of} {\it Proposition \ref{smoothing}} Without loss of generality we may assume that $0<y<1$. The case $y=1$ is 
considered in \cite{GT:11}.
The proof of Proposition \ref{smoothing} is an adaption  of  the proof of Proposition 4.1 
from  \cite{GT:11}. We provide it here for completeness.
By Lemma \ref{supremum} in Appendix, we have
\begin{equation}\label{gav1}
\sup_x|F(x)-\widetilde G_y(x)|\le\sup_{|x|\in\mathcal J_{\varepsilon}'}|F(x)-
\widetilde G_y(x)|+\frac{2}{\pi\sqrt{y(1-y)}}\4\varepsilon^{\frac32}. 
\end{equation}
Let $x\in\mathbb J_{\varepsilon}$. Recall that $\gamma=\min\{|x|-1+\sqrt y,1+\sqrt y-|x|\}$. 
Then,  according to condition \eqref{avcond} we have
$x+\frac{vH}{\sqrt{\gamma}}\in\mathbb J'_{\varepsilon}$.
Denote by $v'=\frac v{\sqrt{\gamma}}$. For any $x\in \mathbb J'_{\varepsilon}$, we have
\begin{align}
 \Big|\frac1{\pi}&\im\Big(\int_{-\infty}^x(S_F(u+iv')-S_{\widetilde G_y}(u+iv'))du\Big)
\Big|\notag\\&\ge \frac1{\pi}
\im\Big(\int_{-\infty}^x(S_F(u+iv')-S_{\widetilde G_y}(u+iv'))du\Big)
\notag\\&=\frac1{\pi}\int_{-\infty}^x\left[\int_{-\infty}^{\infty}
\frac{v'd(F(t)-\widetilde G_y(t))}{(t-u)^2+{v'}^2}\right]du\notag\\&=
\frac1{\pi}\int_{-\infty}^x\left[\int_{-\infty}^{\infty}
\frac{2v'(t-u)(F(t)-\widetilde G_y(t))dt}{((t-u)^2+{v'}^2)^2}\right]\notag\\&=
\frac1{\pi}\int_{-\infty}^{\infty}(F(t)-\widetilde G_y(t))\left[\int_{-\infty}^x
\frac{2v'(t-u)du}{((t-u)^2+{v'}^2)^2}dt\right]\notag\\&=
\frac1{\pi}\int_{-\infty}^{\infty}\frac{(F(x-v't)-\widetilde G_y(x-v't))dt}{t^2+1}.
\end{align}
Since $F$ is non decreasing, we obtain
\begin{align}
 \frac1{\pi}&\int_{|t|\le H}\frac{(F(x-v't)-\widetilde G_y(x-v't))dy}{t^2+1}\notag\\&
\ge
{\tau}(F(x-v'H)-\widetilde G_y(x-v'H))-\frac1{\pi}\int_{|t|\le H}|
\widetilde G_y(x-v't)-\widetilde G_y(x-v'H))|dt\notag\\&\ge
\tau(F(x-v'H)-\widetilde G_y(x-v'H))-\frac1{v'\pi}\int_{|t|\le v'H}|\widetilde G_y(x-t)-
\widetilde G_y(x-v'H))|dt .
\end{align}
Moreover, by inequality \eqref{gav1}, we have
\begin{align}
\Big| \frac1{\pi}\int_{|t|> H}&\frac{(F(x-v't)-\widetilde G_y(x-v't))dy}{t^2+1}\Big|
\le (1-\tau)\Delta(F,\widetilde G_y).
\end{align}
 Let   $\Delta_{\varepsilon}(F,\widetilde G_y)=\sup_{x\in\mathbb J_{\varepsilon}}
|F(x)-\widetilde G_y(x)|$ and let
$x_n\in\mathbb J_{\varepsilon}$ such that
$F(x_n)-\widetilde G_y(x_n)\to\Delta_{\varepsilon}(F,\widetilde G)$. Then $x_n'=x_n+v'a
\in\mathbb J'_{\varepsilon}$. We have
\begin{align}\label{gav2}
 \sup_{x\in\mathbb J'_{\varepsilon}}&\left|\im\int_{-\infty}^x(S_F(u+iv')
-S_{\widetilde G_y}(u+iv'))du\right
|\ge \tau(F(x_n)-\widetilde G_y(x_n))\notag\\&-\frac1{\pi v}
\sup_{x\in\mathbb J'_{\varepsilon}}\sqrt{\gamma}\int_{|t|<2v'H}|\widetilde G_y(x+t)
-\widetilde G_y(x)|dt
-(1-\tau)\Delta(F,\widetilde G_y).
\end{align}
Furthermore, assume for the definitely that $t\ge 0$. Using Lemma \ref{densitymp} 
in Appendix, we get
\begin{equation}
 |\widetilde G_y(x+t)-\widetilde G_y(x)|\le |t|\sup_{u\in[x,x+t]}\widetilde G'_y(u)
\le \frac{2|t|\sqrt{\gamma+t}}{\pi\sqrt{1-\sqrt y}}
\le \frac{2|t|\sqrt{\gamma+\varepsilon}}{\pi\sqrt{y(1-\sqrt y)}},
\end{equation}
for $|t|\le 2v'H\le \varepsilon$.
This implies after integrating
\begin{align}\label{gav3}
\frac1{\pi v} \sup_{x\in\mathbb J'_{\varepsilon}}&\sqrt{\gamma}\int_{|t|<2v'H}|
\widetilde G_y(x+t)-\widetilde G_y(x)|dt\notag\\&
\le \frac {2H^2v}{\pi^2\sqrt{y(1-\sqrt y)}} \sup_{x\in\mathbb J'_{\varepsilon}}
\frac{\sqrt{\gamma+\varepsilon}}{\sqrt{\gamma}}\le 
\frac {2H^2\sqrt 3v}{\pi^2\sqrt{y(1-\sqrt y)}}.
\end{align}
We use here that for $|x|\in\mathcal J'_{\varepsilon}$ the inequality 
$\gamma\ge\frac12\varepsilon$ holds.
Inequalities (\ref{gav1}), (\ref{gav2}) and (\ref{gav3}) together imply
\begin{align}\label{Deltaeps}
 \sup_{x\in\mathbb J'_{\varepsilon}}&\left|\im\int_{-\infty}^x(S_F(u+iv')
-S_{\widetilde G_y}(u+iv'))du\right|\notag\\
&\ge(2\tau-1)\Delta(F,\widetilde G_y)-\frac12C_1v-(1-\tau)C_2\varepsilon^{\frac32},
\end{align}
where $C_1=\frac {2H^2\sqrt 3}{\pi^2\sqrt{y(1-\sqrt y)}}$ and 
$C_2=\frac{2}{\pi\sqrt{y(1-\sqrt y)}}$.
Similar arguments may be used to prove this inequality 
in case  there is a sequence $x_n\in\mathbb J_{\varepsilon}$ such $F(x_n)-G(x_n)\to
-\Delta_{\varepsilon}(F,G)$. 
In view of \eqref{Deltaeps} and $2\alpha-1=1/2$ this completes the proof.
\end{Proof of}
\subsection{Some inequalities for  the Stieltjes transform of Marchenko - Pastur distribution}

\begin{lem}\label{stmp} For Stieltjes transform $\widetilde S_y(z)$ 
of symmetrizing Marchenko--Pastur distribution
 the following inequalities hold
\begin{align}\label{new1}
 |S_y(z)|\le \frac1{\sqrt y},\quad |z+\frac{y-1}z+yS_y(z)|\ge {\sqrt y}.
\end{align}
 Moreover,
\begin{equation}\label{new2}
|z+yS_y(z)|\ge C_1(y)=\begin{cases}1,\text{ for }y=1,\\
\frac1{1+2\sqrt y},\text{ for } 0<y<1.\end{cases}.
\end{equation}
\end{lem}
\begin{proof}Let $\widehat S_y(z)=\frac{-z-\frac{y-1}z-\sqrt{(z+\frac{y-1}z)^2-4y}}{2y}$. 
Note that $S_y(z)$ and $\widehat S_y(z)$ are the roots of equation
\begin{equation}
yS_y(z)^2+(z+\frac{y-1}z)S_y(z)+1=0.
\end{equation}
From here it follows
\begin{equation}\label{equa120}
 |S_y(z)||\widehat S_y(z)|=\frac1y.
\end{equation}
Similar to \cite{Bai:93}, Section 3, we note that
\begin{equation}
 \text{\rm sign}\re\{z+\frac {y-1}z\}=\text{\rm sign}\{\re\sqrt{(z+\frac{y-1}z)^2-4y}\}.
\end{equation}
(For more details see \cite{Bai:93}, pp. 631--632.)
 This implies that
\begin{equation}\label{ineq120}
 |S_y(z)|\le|\widehat S_y(z)|.
\end{equation}
Inequality \eqref{ineq120} and equality \eqref{equa120} together imply \eqref{new1}. To prove \eqref{new2} introduce notation
\begin{align}
w _y(z):==y+zS_y(z).
\end{align}
It straightforward to check 
\begin{equation}\label{new0}
\frac1{ w_y(z)}=-yS_y(z)-\frac{y-1}{z}.
\end{equation}
Relations \eqref{new1} imply, for $0<y<1$ and $1-\sqrt y\le |z|\le 1+\sqrt y$,
\begin{equation}
|-yS_y(z)-\frac{y-1}{z}|\le \sqrt y+\Big|\frac{1-y}z\Big|\le 1+2\sqrt y.
\end{equation}
The last inequality and relation implies, that for 
$1-\sqrt y\le |z|\le\sqrt{1-y}$ and for $0<y<1$,
\begin{equation}
|w_y(z)|\ge \frac1{1+2\sqrt y}.
\end{equation}
For $y=1$,
\begin{equation}
|w_y(z)|\ge 1.
\end{equation}
Thus Lemma \ref{stmp} is proved.
\end{proof}
\begin{lem}\label{sign}
 For $z=u+iv$ with $1-\sqrt y\le|u|\le 1+\sqrt y$ and $v>0$ the following relation holds
\begin{equation}
 \re\{(z+\frac{y-1}z)^2-4y\}\le0.
\end{equation}
\end{lem}
\begin{proof}
 It straightforward to check
\begin{align}\label{re1}
A:= \re\{(z+\frac{y-1}z)^2-4y\}=u^2(1-\frac{(1-y)}{|z|^2})^2-4 y-v^2(1+\frac{1-y}{|z|^2})^2.
\end{align}
We rewrite this equality as 
\begin{equation}
 A=(u^2-v^2)(1+\frac{(1-y)^2}{|z|^4})-2(1+y)\le u^2+\frac{(1-y)^2}{|u|^2}-2(1+y).
\end{equation}
Denote by $t=u^2$ and consider the equation
\begin{equation}
 t^2-2(1+y)t+(1-y)^2=0.
\end{equation}
Solving it, we find
\begin{equation}
 t_{1,2}=(1+y)\pm\sqrt{4y}=(1\pm\sqrt y)^2.
\end{equation}
This immediately implies that $A\le 0$, for $1-\sqrt y\le |u|\le 1+\sqrt y$.
Thus Lemma \ref{sign} is proved.
\end{proof}
\begin{lem}\label{imsqrt}
 For any $z=u+iv$ with $1-\sqrt y\le|u|\le 1+\sqrt y$,  the following inequality holds
\begin{equation}
 \im{\sqrt{(z+\frac{y-1}z)^2-4y}}\ge\frac12|\sqrt{(z+\frac{y-1}z)^2-4y}|
\ge \frac{y^{\frac 14}}{2}\sqrt{\gamma+v}. 
\end{equation}
 \end{lem}
\begin{proof}
 By Lemma \ref{sign}, for $z=u+iv$ with $1-\sqrt y\le|u|\le 1+\sqrt y$, 
we get $\re((z+\frac{y-1}z)^2-4y)\le0$ and 
$\frac{\pi}2\le\arg((z+\frac{y-1}z)^2-4y)\le \frac{3\pi}2$. 
Therefore, 
\begin{equation}\label{error30} 
 \im\{\sqrt{(z+\frac{y-1}z)^2-4y}\}\ge\frac1{\sqrt 2}|(z+\frac{y-1}z)^2-4y|^{\frac12}. 
\end{equation} 
Furthermore, we have
\begin{equation}
 (z+\frac{y-1}z)^2-4y=\frac{(z+1+\sqrt y)(z+1-\sqrt y)(z-1+\sqrt y)(z-1-\sqrt y)}{z^2}.
\end{equation}
Let $u=-1-\sqrt y+\gamma$. Then, $|z|\ge 1$ and 
\begin{align}\label{tr1}
 |(z+\frac{y-1}z&)^2-4y|\ge \frac1{2|z|^2}(\gamma+v)|z-1+\sqrt y||z+1-\sqrt y||z-1-\sqrt y|
\notag\\&\ge \frac1{2|z|^2}(\gamma+v)|-2+\gamma+iv||-2\sqrt y+\gamma+iv||-2-2\sqrt y+\gamma+iv|.
\end{align}
Note that
\begin{equation}
 \frac{|-2+\gamma+iv|}{|z|}=\sqrt{\frac{(2-\gamma)^2+v^2}{(1+\sqrt y-\gamma)^2+v^2}}\ge 1,
\end{equation}
and
\begin{equation}
 \frac{|-2-2\sqrt y+\gamma+iv|}{|z|}=\sqrt{\frac{(2+2\sqrt y-\gamma)^2+v^2}
{(1+\sqrt y-\gamma)^2+v^2}}\ge 1.
\end{equation}
These inequalities together imply
\begin{equation}\label{tr10}
 |(z+\frac{y-1}z)^2-4y|\ge \frac{\sqrt y}2(\gamma+v).
\end{equation}
For $u=-1+\sqrt y-\gamma$, we have $|z|\ge 1-\sqrt y+\gamma$ and
\begin{equation}
 |(z+\frac{y-1}z)^2-4y|\ge \frac1{2|z|^2} (\gamma+v)|2\sqrt y-\gamma+iv||-2
+2\sqrt y-\gamma+iv||-2-\gamma+iv|.
\end{equation}
Note that
\begin{equation}
\frac{(2(1-\sqrt y)+\gamma)^2+v^2}{(1-\sqrt y+\gamma)^2+v^2}\ge 1,
\end{equation}
and
\begin{equation}
\frac{(2+\gamma)^2+v^2}{(1-\sqrt y+\gamma)^2+v^2}\ge 1.
\end{equation}
These inequalities imply 
\begin{equation}
|(z+\frac{y-1}z)^2-4y|\ge \frac{\sqrt y}2(\gamma+v). 
\end{equation}
Similar we consider $u=1-\sqrt y+\gamma$ and $u=1+\sqrt y-\gamma$.
Thus Lemma \ref{imsqrt} is proved.
\end{proof}
\subsection{The estimations of of the error terms}
 We shall use McDiarmid inequality for martingales and sums of independent random variables in the following form
\begin{lem}\label{diarmid}
Let $\textfrak N_0=\{\emptyset,\Omega\}\subset\textfrak N_1\subset\cdots\subset
\textfrak N_n\subset\textfrak M$ be a family of
 sub-$\sigma$-algebras of the measurable space $\{\Omega,\textfrak M\}$ and 
let $M_n=\xi_1+\ldots +\xi_n$ be a martingale with bounded differences
$\xi_j=M_j-M_{j-1}$ such that
$$
\Pr\{|\xi_j|\le b_j\}=1, \quad\text{for}\quad j=1,\ldots,n.
$$
Then, for $x>\sqrt8$,
\begin{equation}
 \Pr\Big\{|M_n|\ge x\Big\}\le c(1-\Phi(\frac x{\sigma}))=\int_{\frac x{\sigma}}^{\infty}
\varphi(t)dt,
\quad\varphi(t)=\frac1{\sqrt{2\pi}}\exp\{\frac{-t^2}2\}
\end{equation}
with some numerical constant $c>0$ and $\sigma^2=b_1^2+\cdots+b_n^2.$
\end{lem}

\begin{proof}
 The result follows from  Theorem 1.1 \cite{Bentkus:2007}.
\end{proof}\begin{prop}\label{epsilon}Under  the conditions of  Theorem \ref{main} there exist absolute constants $C$ and $c$ such that, for $v\ge v_0$ and $j=1,\ldots,n+p$, 
\begin{equation}\notag
\Pr\{|\varepsilon_j|\ge \frac{C \lna^{2+\frac2{\varkappa}}\sqrt{\im m_n^{(j)}(z)+\frac1{nv}}}{\sqrt{nv}}\}\le\exp\{-c\lna^2\}
\end{equation}
\end{prop}
\begin{Proof of} {\it Proposition \ref{epsilon}} We divide the proof into three parts, which formulate as Lemmas.
     We start with $\varepsilon_{j1}$, for $j=1,\ldots,n$. Recall that
     \begin{equation}\notag
      \varepsilon_{j1}=\frac1p\sum_{l=1}^p(X_{jl}^2-1)R^{(j)}_{l+n,l+n}.
     \end{equation}
     \begin{lem}\label{eps1*}Under  the conditions of  Theorem \ref{main} there exist absolute constants $C$ and $c$ such that, for $v\ge v_0$ and $j=1,\ldots,n+p$, 
\begin{equation}\notag
\Pr\{|\varepsilon_{j1}|\ge \frac{C \lna^{\frac2{\varkappa}+2}\sqrt{\im m_n^{(j)}(z)+\frac1{nv}}}{\sqrt{nv}}\}\le\exp\{-c\lna^2\}
\end{equation}
     \end{lem}
     \begin{proof}
Note that $\xi_l=X_{jl}^2-1$ are independent for $l=1,\ldots,p$ and $|\xi_l|\le C\lna^{\frac2{\varkappa}}$. Moreover, $\xi_l$ are independent on $R^{(j)}_{l+n,l+n}$.
Conditioning $\mathfrak M^{(j)}$ and applying MxDiarmid inequality(see Lemma \ref{diarmid} in Appendix) with $\sigma^2=\frac{C\lna^{\frac4{\varkappa}}\im m_n^{(j)}(z)}{nv}$ and
$x=\frac{C\lna^{\frac2{\varkappa}+2}\im^{\frac12} m_n^{(j)}(z)}{\sqrt{nv}}$, we get
\begin{equation}\label{eps1*}
\Pr\{|\varepsilon_{j1}|\ge \frac{C \lna^{\frac2{\varkappa}+2}\sqrt{\im m_n^{(j)}(z)+\frac1{nv}}}{\sqrt{nv}}\}\le\exp\{-c\lna^2\}.
\end{equation}
Similar we get the bound for $\varepsilon_{l+n,1}$, for $l=1,\ldots,p$,
\begin{equation}\label{eps2}
\Pr\{|\varepsilon_{l+n,1}|\ge \frac{C \lna^{\frac2{\varkappa}+2}\sqrt{\im m_n^{(j)}(z)+\frac1{nv}}}{\sqrt{nv}}\}\le\exp\{-c\lna^2\}.
\end{equation}
\end{proof}
Furthermore, consider $\varepsilon_{j2}$, for $j=1,\ldots,n$.
Recall that
\begin{equation}\notag
 \varepsilon_{j2}=\frac1p\sum_{1\le r\ne l\le p}X_{jl}X_{jr}R^{(j)}_{r+n,l+n}.
\end{equation}
\begin{lem}\label{eps2*} Under  the conditions of  Theorem \ref{main} there exist absolute constants $C$ and $c$ such that, for $v\ge v_0$ and $j=1,\ldots,n+p$, 
\begin{equation}\notag
\Pr\{|\varepsilon_{j2}|\ge \frac{C \lna^{\frac2{\varkappa}+2}\sqrt{\im m_n^{(j)}(z)+\frac1{nv}}}{\sqrt{nv}}\}\le\exp\{-c\lna^2\}
\end{equation}
\end{lem}
\begin{proof}
Conditioning $\mathfrak M^{(j)}$, we may represent $\varepsilon_{j2}$ as martingale
\begin{equation}\notag
 \varepsilon_{j2}=\sum_{t=1}^pY_t,
\end{equation}
with martingale-difference
\begin{equation}\notag
 Y_t=\frac2pX_{jt}\sum_{r=1}^{t-1}X_{jr}R^{(j)}_{q+n,t+n}.
\end{equation}
Conditioning $\mathfrak M^{(j)}$ and applying McDiarmid's inequality (see Lemma \ref{diarmid} in the Appendix) , we get
\begin{equation}\notag
\Pr\{|Y_t|\le    \frac{C\lna^{\frac2{\varkappa}+1}(\sum_{r=1}^p|R^{(j)}_{r+n,t+n}|^2)^{\frac12}}{p} \}\le\exp\{-c\lna^2\}                                                                              
\end{equation}
Truncating $Y_t$ at the level $\frac{C\lna^{\frac2{\varkappa}}(\sum_{r=1}^p|R^{(j)}_{r+n,t+n}|^2)^{\frac12}}{p}$ and applying McDiarmid's inequality again
(see Lemma \ref{diarmid} in the Appendix), we get
\begin{equation}\label{eps3}
 \Pr\{|\varepsilon_{j2}|>\frac{C\lna^{\frac2{\varkappa}+2}(\im  m_n^{(j)}(z)+\frac1{nv})^{\frac12}}{\sqrt{nv}}\}\le\exp\{-c\lna^2\}.
\end{equation}
Similar we get the bound for $\varepsilon_{l+n,2}$, for $l=1,\ldots,p$,
\begin{equation}\label{eps4}
 \Pr\{|\varepsilon_{l+n,2}|>\frac{C\lna^{2{\varkappa}+2}(\im m_n^{(j)}(z)+\frac 1{nv})^{\frac12}}{\sqrt{nv}}\}\le\exp\{-c\lna^2\}.
\end{equation}
\end{proof}
To finish the proof we consider $\varepsilon_{j3}$.

    \begin{lem}\label{lem2.3}
Under the conditions of  Theorem \ref{main} we have, for any $z=u+iv$ with $u\in\mathbb R$, 
$v>0$, and for any $j=1,\ldots,n$,
\begin{equation}\notag
|\varepsilon_{j3}|\le \frac y{nv}.
\end{equation}
\end{lem}
\begin{proof}It is straightforward to check that
\begin{equation}
 \sum_{l=1}^pR_{l+n,l+n}=n m_n(z)-\frac{p-n}{z}=\sum_{l=1}^{p+n}R_{ll}-n m_n(z)
\end{equation}
and
\begin{equation}
 \sum_{l=1}^pR^{(j)}_{l+n,l+n}=\sum_{l=1, l\ne j}^{p+n}R^{(j)}_{ll}-n m_n^{(j)}(z).
\end{equation}
Furthermore,
\begin{align}\label{m2}
 n m_n(z)=\frac12\Tr\mathbf R+\frac{p-n}{2z},\quad
n {m_n^{(j)}}(z)=\frac12\Tr\mathbf R^{(j)}+\frac{p-n+1}{2z}.
\end{align}
This implies
\begin{equation}\label{e3}
  \sum_{l=1}^pR_{l+n,l+n}- \sum_{l=1}^pR^{(j)}_{l+n,l+n}=\frac12\Tr\mathbf R
-\frac12\Tr\mathbf R^{(j)}-\frac1{2z}.
\end{equation}

The conclusion of Lemma \ref{lem2.3} follows immediately from the inequality 
\newline$|\Tr\mathbf R-\Tr\mathbf R^{(j)}|\le v^{-1}$ and $|\frac1z|\le v^{-1}$
(see Lemma 4.1 in \cite{GT:03}).
\end{proof}
It is straightforward to check now  that, for $j=1,\ldots,n+p$,
\begin{equation}\label{eps5}
 |\varepsilon_{j3}|\le \max\{ \frac{\im^{\frac12}m_n^{(j)}(z)}{\sqrt{nv}},\frac1{nv}\} \text{ a.s.}
\end{equation}

Combining inequalities \eqref{eps1} -- \eqref{eps5} we get the claim.
Thus Proposition \ref{epsilon} is proved.
    \end{Proof of}
\begin{rem}
Equalities \eqref{m2} imply that
\begin{equation}\label{mnj}
 |m_n(z)-m_n^{(j)}(z)|\le (nv)^{-1}.
\end{equation}
\end{rem}

\begin{prop}\label{momeps}
 Under conditions of theorem \ref{main} there exist  constants  $C$ and $A$ depending on $\varkappa$ and $y$ only such that, for any $q\le A\log n$,
 \begin{equation}
  \E\{|\varepsilon_j|^q\Big|\mathfrak M^{(j)}\}
  \le \frac{C^qq^q\lna^{\frac {2q}{\varkappa}}(\im m_n^{(j)}+\frac1{nv})^{\frac q2}}{(nv)^{\frac q2}}.
 \end{equation}

\end{prop}
\begin{proof}
 First we note that
 \begin{equation}
  \E\{|\varepsilon_j|^q\Big|\mathfrak M^{(j)}\}\le 3^q(\E\{|\varepsilon_{j1}|^q\Big|\mathfrak M^{(j)}\}+\E\{|\varepsilon_{j2}|^q\Big|\mathfrak M^{(j)}\}+\E\{|\varepsilon_{j3}|^q\Big|\mathfrak M^{(j)}\})
 \end{equation}
We divide now the proof into three Lemmas.
\begin{lem}\label{eps1} Under conditions of theorem \ref{main} there exist  constants  $C$ and $A$ depending on $\varkappa$ and $y$ only such that, for any $q\le A\log n$, and for all $j=1,\ldots,n$,
\begin{equation}
 \E\{|\varepsilon_{j1}|^q\Big|\mathfrak M^{(j)}\}\le\frac{C^qq^q\lna^{\frac{2q}{\varkappa}}(\im (m_n^{(j)}+\frac{y-1}z))^{\frac q2}}{(nv)^{\frac q2}}.
 \end{equation}
\end{lem}
\begin{Proof of} {\it Lemma}Recall that
\begin{equation}
 \varepsilon_{j1}=\frac1p\sum_{l=1}^p(X_{jl}^2-1)R^{(j)}_{l+n,l+n}.
\end{equation}
Conditioning on $\mathfrak M^{(j)}$ and applying Rosethal's inequality, we obtain
\begin{align}\label{1+}
 \E\{|\varepsilon_{j1}|^q\Big|\mathfrak M^{(j)}\}\le C^qq^qp^{-q}((\sum_{l=1}^p|R^{(j)}_{l+n,l+n}|^2)^{\frac q2}+\sum_{l=1}^p|R^{(j)}_{l+n,l+n}|^q\E|(X_{jl}^2-1|^q).
\end{align}
According to Remark \ref{trunc00},
\begin{equation}\label{2+}
 \E|(X_{jl}^2-1|^q\le \lna^{\frac{2q}{\varkappa}}.
\end{equation}
Note that
\begin{equation}\label{3+}
 \sum_{l=1}^p|R^{(j)}_{l+n,l+n}|^2\le \Tr|\mathbf R^{(j)}|^2\le \frac{2p}v\im(ym_n^{(j)}(z)+\frac{y-1}z),
\end{equation}
and
\begin{equation}\label{4+}
 \sum_{l=1}^p|R^{(j)}_{l+n,l+n}|^q\le (\sum_{l=1}^p|R^{(j)}_{l+n,l+n}|^2)^{\frac q2}.
\end{equation}
Combining relations \eqref{1+} -- \eqref{4+}, we get 

\begin{equation}
 \E\{|\varepsilon_{j1}|^q\Big|\mathfrak M^{(j)}\}\le\frac{C^qq^q\lna^{\frac{2q}{\varkappa}}(\im (m_n^{(j)}+\frac{y-1}z))^{\frac q2}}{(nv)^{\frac q2}}.
 \end{equation}
 Thus lemma \ref{eps1} is proved.
\end{Proof of}
\begin{lem}\label{eps2}
 Under conditions of theorem \ref{main} there exist  constants  $C$ and $A$ depending on $\varkappa$ and $y$ only such that, for any $q\le A\log n$, and for all $j=1,\ldots,n$,
\begin{equation}
 \E\{|\varepsilon_{j2}|^q\Big|\mathfrak M^{(j)}\}\le\frac{C^qq^q\lna^{\frac {2q}{\varkappa}}(\im (m_n^{(j)}+\frac{y-1}z))^{\frac q2}}{(nv)^{\frac q2}}.
 \end{equation}
\end{lem}
\begin{Proof of} {\it Lemma.}
Recall that
 \begin{equation}
 \varepsilon_{j2}=\frac1p\sum_{1\le l\ne k\le p}X_{jl}X_{jk}R^{(j)}_{l+n,k+n}.
\end{equation}
Conditioning on $\mathfrak M^{(j)}$ and applying Burkholder's inequality (see \cite{Rosenthal:1970} and \cite{Johnson:1985}), we obtain
\begin{align}\label{874}
 \E\{|\varepsilon_{j2}|^q\Big|\mathfrak M^{(j)}\}&\le C^qq^qp^{-q}\Big(\E\{(\sum_{l=2}^p|\sum_{k=1}^{l-1}X_{jk}R^{(j)}_{l+n,k+n}|^2)^{\frac q2}\Big|\mathfrak M^{(j)}\}\notag\\&\qquad\qquad+
 \sum_{k=2}^l\E\{|\sum_{k=1}^{l-1}X_{jk}R^{(j)}_{l+n,k+n}|^q\Big|\mathfrak M^{(j)}\}\E|X_{jl}|^q\Big).
\end{align}
To estimate the second sum in the r.h.s of \eqref{874} we apply Rosenthal's inequality as well. We get
\begin{equation}\label{eta22}
\sum_l\E|\sum_{k}X_{jk}b^{(j)}_{lk}|^{q}\E|X_{11}|^{q}\le C^{q}q^{q} \lna^{\frac q{\varkappa}}
\Big(\sum_{l}\big(\sum_{k}|b^{(j)}_{lk}|^2\big)^{\frac q2}+
\lna^{\frac q{\varkappa}}\sum_{l}\sum_{k}|b^{(j)}_{lk}|^{q}\Big),
\end{equation}
where $b^{(j)}_{lk}:=R^{(j)}_{l+n,k+n}$
Consider the random variables
$Y_k=\frac1{c\lna^{\frac1{\varkappa}}}X_{jk}$.
Note that $Y_1,\ldots, Y_p$ are independent and, by Remark \ref{trunc00},
$|Y_{k}|\le 1,\quad \E Y_k=0$.

Consider the quadratic form in $Y_1,\ldots,Y_p$
\begin{equation}\notag
 f(Y_1,\ldots,Y_n)=\sum_{l=1}^p(\sum_{k=1}^lb^{(j)}_{lk}Y_k)^2
\end{equation} 
Note that $f$ is a convex function.
Let $Z_1,\ldots, Z_p$ denote standard Gaussian r.v.'s.
Then it follows from results of Bobkov \cite{bobkov:96}, \cite{bobkov:00}  (Choquet comparison of measures), that
\begin{equation}\notag
 \E^{\frac1m}|f(Y_1,\ldots,Y_p)|^m\le \E^{\frac1m}|f(c_0Z_1,\ldots,c_0Z_p)|^m,
\end{equation}
were $c_0=\frac{\sqrt{2\pi}}2$.  Note that 
\begin{equation}\notag
 f(c_0Z_1,\ldots,c_0Z_p)=c_0^2f(Z_1,\ldots,Z_p).
\end{equation}
For the Gaussian r.v.'s we have (\cite{bobkovgoetze:99}, Theorem 3.1)
\begin{equation}\notag
 \E^{\frac1m}|f(Z_1,\ldots,Z_p)|^m\le Cm\E|f(Z_1,\ldots,Z_p)|=Cm\sum_{l=1}^p\sum_{k=1}^p|b^{(j)}_{lk}|^2.
\end{equation}
In our case 
\begin{equation}
\frac1p\sum_{l=1}^p\sum_{r=1}^p|b^{(j)}_{lk}|^2\le v^{-1}\im (m_n^{(j)}(z)+\frac{y-1}z).
\end{equation}
Applying these inequalities, we get, using that $X_{jq}=c\lna^{\frac1{\varkappa}}Y_{q}$, 
\begin{align}\label{bobkov}
 \frac1{n^{\frac q2}}\Big(\E\{\Big(\frac1n\sum_{l=1}^p(
 |\sum_{k}X_{jr}b^{(j)}_{lk}|^{2}\Big)^{\frac q2}\big|\mathfrak M^{(j)}\}\Big)\le
 Cq^qv^{-\frac q2}\lna^{\frac2{\varkappa}}(\im (m_n^{(j)}(z)+\frac{y-1}z))^q.
\end{align}
Combining inequalities \eqref{eta21}, \eqref{eta22} and \eqref{bobkov}, we get
\begin{equation}
\E\{|\varepsilon_{j2}|^{q}\Big|\mathfrak M^{(j)}\}\le \frac{C^{q}q^{q}\lna^{\frac{2q}{\varkappa}}}{(nv)^{\frac q2}}(\im (m_n^{(j)}(z)+\frac{y-1}z))^{\frac q2}.
\end{equation}
Thus Lemma \ref{eps2} is proved.
\end{Proof of}
\begin{lem}\label{eps3}
Under conditions of theorem \ref{main} there exist  constants  $C$ and $A$ depending on $\varkappa$ and $y$ only such that, for any $q\le A\log n$, and for all $j=1,\ldots,n$,
\begin{align}
 \E\{|\varepsilon_{j3}|^q\Big|\mathfrak M^{(j)}\}&\le\frac{C^q}{n^q|z|^q}+\frac{C^q}{n^q}\E\{|R_{jj}|^q\Big|\mathfrak M^{(j)}\}\notag\\&
 +\frac{C^qq^{\frac {3}2}\lna^{\frac {2q}{\varkappa}}}{(nv)^{\frac {3q}2}}(\im(m_n^{(j)}(z)+\frac{y-1}{z}))^{\frac q2}\E^{\frac12}\{|R_{jj}|^{2q}\Big|\mathfrak M^{(j)}\}.
 \end{align}
 
\end{lem}
\begin{rem}
 For $z=u+iv$ with $v\ge v_0$, we have
 \begin{align}
 \E\{|\varepsilon_{j3}|^q\Big|\mathfrak M^{(j)}\}&\le\frac{C^q}{n^q|z|^q}+\frac{C^q}{n^q}\E\{|R_{jj}|^q\Big|\mathfrak M^{(j)}\}\notag\\&
 +\frac{C^q\beta_n^q}{(nv)^q}(\im(m_n^{(j)}(z)+\frac{y-1}{z}))^{\frac q2}\E^{\frac12}\{|R_{jj}|^{2q}\Big|\mathfrak M^{(j)}\}.
 \end{align}
\end{rem}

\begin{Proof of} {\it Lemma.}
 Note that
 \begin{equation}
  \varepsilon_{j3}=\frac1{2p}(\Tr\mathbf R-\Tr\mathbf R^{(j)})+\frac1{2pz}.
 \end{equation}
 From Shur's decomposition formula it follows that
 \begin{equation}
  \Tr\mathbf R-\Tr\mathbf R^{(j)}=(1+\eta_j)R_{jj},
 \end{equation}
 where
 \begin{equation}
  \eta_j=\frac1p\sum_{1\le l,k\le p}X_{jl}X_{jk}[(\mathbf R^{(j)})^2]_{l+n,k+n}.
 \end{equation}
This implies that
\begin{equation}
 |\varepsilon_{j3}|^q\le \frac{(3y)^q}{n^q}(|R_{jj}|^q+|\eta_j|^q|R_{jj}|^q+\frac1{|z|^q}).
\end{equation}

Applying Cauchy - Schwartz inequality, we obtain
\begin{align}
 \E\{|\varepsilon_{j3}|^q\Big|\mathfrak M^{(j)}\}\le\frac{C^q}{n^q|z|^q}+\frac{C^q}{n^q}\E\{|R_{jj}|^q\Big|\mathfrak M^{(j)}\}+\frac{(3y)^q}{n^q}\E^{\frac12}\{|\eta_j|^{2q}\Big|\mathfrak M^{(j)}\}
 \E^{\frac12}\{|R_{jj}|^{2q}\Big|\mathfrak M^{(j)}\}.
\end{align}
Using Lemma \ref{eta} below, we get 
\begin{align}
  \E\{|\varepsilon_{j3}|^q\Big|\mathfrak M^{(j)}\}&\le\frac{C^q}{n^q|z|^q}+\frac{C^q}{n^q}\E\{|R_{jj}|^q\Big|\mathfrak M^{(j)}\}+\frac{C^qq^q(\im( m_n^{(j)}(z)+\frac{1-y}z))^{\frac q2}}{(nv)^{q}}\notag\\&
+\frac{C^qq^{\frac{3q}2}(\im( m_n^{(j)}(z)+\frac{1-y}z))^{\frac q2}}{(nv)^{\frac{3q}2}}\E^{\frac12}\{|R_{jj}|^{2q}\Big|\mathfrak M^{(j)}\}.
\end{align}

\end{Proof of}

Thus lemmas \ref{eps1}, \ref{eps2}, \ref{eps3} together yield Proposition \ref{momeps}. 

\end{proof}

Let
\begin{equation}
\eta_j=\frac1p\sum_{1\le l\ne k\le p}X_{jl}X_{jk}b_{lk}^{(j)},
\end{equation}
where
\begin{equation}
b_{lk}^{(j)}=[{\mathbf R^{(j)}}^2]_{l+n,k+n}.
\end{equation}
\begin{lem}\label{eta}Under  the conditions of  Theorem \ref{main} there exist absolute constants $C$ and $c$ such that, for $z\in\mathbb G$ and $j=1,\ldots,n$, 
\begin{equation}\notag
\E\{|\eta_j|^{2q}\Big|\mathfrak M^{(j)}\}\le\frac{C^q(\im m_n^{(j)}(z))^{2q}}{v^{2q}}+\frac{C^{2q}(1-y)^{2q}}{|z|^{2q}}
+\frac{C^qq^{3q}\lna^{\frac {2q}{\varkappa}}(\im( m_n^{(j)}(z)+\frac{y-1}z))^{q}}{n^qv^{3q}}.
\end{equation}
\end{lem}
\begin{proof}
We may represent $\eta_j$ in the form
\begin{equation}
\eta_j=\eta_{j0}+\eta_{j1}+\eta_{j2},
\end{equation}
where
\begin{equation}
\eta_{j0}=\frac1p\sum_{l=1}^pb_{ll}^{(j)}, \quad\eta_{j1}=\frac1p\sum_{l=1}^p(X_{jl}^2-1)b_{ll}^{(j)},\notag\\
\eta_{j2}=\frac1p\sum_{1\le l\ne k\le p}X_{jl}X_{jk}b_{lk}^{(j)}.
\end{equation}
First we note
\begin{equation}
|\eta_{j0}|\le \frac1p\sum_{l=1}^p|[{\mathbf R^{(j)}}^2]_{l+n,l+n}|\le \frac1p
\sum_{l=1, l\ne j}^{n+p}\sum_{k=1, k\ne j}^{n+p}|R^{(j)}_{kl}|^2\le Cv^{-1}\im m_n^{(j)}(z)+\frac{C(1-y)}{|z|^2}.
\end{equation}
Furthermore, conditioning on $\mathfrak M^{(j)}$ and applying H\"older's inequality, we get
\begin{align}\label{eta4}
\E\{|\eta_{j1}|^{2q}\Big|\mathfrak M^{(j)}\}\le \frac{C^{2q}q^{2q}}{n^{2q}}(\Big(\sum_{l=1}^p|b_{ll}^{(j)}|^2)^q+\sum_{l=1}^p|b_{ll}^{(j)}|^{2q}\E|X_{11}|^{2q}\Big).
\end{align}
Since $|b^{(j)}_{l+n,k+n}|\le v^{-2}$, we have
\begin{equation}\label{eta10}
\frac1p\sum_{l=1}^p|b_{ll}^{(j)}|^2\le v^{-2}\frac1p\sum_{l=1}^p|b_{ll}^{(j)}|\le v^{-3}\im( m_n^{(j)}(z)+\frac{y-1}z).
\end{equation}
We estimate now the second sum in the right hand side  of \eqref{eta4}.
By Remark \ref{trunc00}, we have
\begin{equation}
\E|X_{11}|^{2q}\le \lna^{\frac{2q}{\varkappa}}.
\end{equation}
Using that, we obtain
\begin{equation}\label{eta11}
\frac1{n^{q}}\sum_{l=1}^p|b_{ll}^{(j)}|^{2q}\E|X_{11}|^{2q}\le \lna^{2q}\Big(\frac1n\sum_{l=1}^p|b_{ll}^{(j)}|^{2}\Big)^q\le 
v^{-3q}\lna^{2q}(\im (m_n^{(j)}(z)+\frac{y-1}z))^q.
\end{equation}
Inequalities \eqref{eta4}, \eqref{eta10} and \eqref{eta11} together imply
\begin{align}\label{eta100}
\E\{|\eta_{j1}|^{2q}\Big|\mathfrak M^{(j)}\}\le \frac{C^{2q}q^{2q}\lna^{2q}}{n^qv^{3q}}\Big(\im (m_n^{(j)}(z)+\frac{y-1}z)\Big)^q.
\end{align}
To bound $\E\{|\eta_{j2}|^{2q}\Big|\mathfrak M^{(j)}\}$ we apply Burkholder's inequality for martingales (see \cite{Hitczenko:1990} and \cite{Burkholder:1973})
We obtain
\begin{align}\label{eta21}
\E\{|\eta_{j2}|^{2q}\Big|\mathfrak M^{(j)}\}\le \frac{C^{2q}q^{2q}}{n^{2q}}\Big(\E\{\Big(\sum_{l}\Big(\sum_{k}X_{jk}b^{(j)}_{lk}\Big)^2\Big)^q+
\sum_l\E|\sum_{k}X_{jk}b^{(j)}_{lk}|^{2q}\E|X_{11}|^{2q}.
\end{align}
To estimate the second sum in the r.h.s of \eqref{eta21} we apply Rosenthal's inequality (see \cite{Rosenthal:1970} and \cite{Johnson:1985}). We get
\begin{equation}\label{eta22}
\sum_l\E|\sum_{k}X_{jk}b^{(j)}_{lk}|^{2q}\E|X_{11}|^{2q}\le C^{2q}q^{2q} \lna^{\frac{2q}{\varkappa}}
\Big(\sum_{l}\big(\sum_{k}|b^{(j)}_{lk}|^2\big)^q+
\lna^{\frac{2q}{\varkappa}}\sum_{l}\sum_{k}|b^{(j)}_{lk}|^{2q}\Big)
\end{equation}
Consider the random variables
$Y_k=\frac1{c\lna^{\frac1{\varkappa}}}X_{jk}$.
Note that $Y_1,\ldots, Y_p$ are independent and, by Remark \ref{trunc00},
$|Y_{k}|\le 1,\quad \E Y_k=0$.

Consider the quadratic form in $Y_1,\ldots,Y_p$
\begin{equation}\notag
 f(Y_1,\ldots,Y_n)=\sum_{l=1}^p(\sum_{k=1}^lb^{(j)}_{lk}Y_k)^2
\end{equation} 
Note that $f$ is a convex function.
Let $Z_1,\ldots, Z_p$ denote standard Gaussian r.v.'s.
Then it follows from results of Bobkov \cite{bobkov:96}, \cite{bobkov:00}  (Choquet comparison of measures), that
\begin{equation}\notag
 \E^{\frac1m}|f(Y_1,\ldots,Y_p)|^m\le \E^{\frac1m}|f(c_0Z_1,\ldots,c_0Z_p)|^m,
\end{equation}
were $c_0=\frac{\sqrt{2\pi}}2$.  Note that 
\begin{equation}\notag
 f(c_0Z_1,\ldots,c_0Z_p)=c_0^2f(Z_1,\ldots,Z_p).
\end{equation}
For the Gaussian r.v.'s we have (\cite{bobkovgoetze:99}, Theorem 3.1)
\begin{equation}\notag
 \E^{\frac1m}|f(Z_1,\ldots,Z_p)|^m\le Cm\E|f(Z_1,\ldots,Z_p)|=Cm\sum_{l=1}^p\sum_{k=1}^p|b^{(j)}_{lk}|^2.
\end{equation}
In our case 
\begin{equation}
\frac1p\sum_{l=1}^p\sum_{r=1}^p|b^{(j)}_{lk}|^2\le v^{-3}\im (m_n^{(j)}(z)+\frac{y-1}z).
\end{equation}
Applying these inequalities, we get, using that $X_{jq}=c\lna^{\frac1{\varkappa}}Y_{q}$, 
\begin{align}\label{bobkov}
 \frac1{n^q}\Big(\E\{\Big(\sum_{l=1}^p(
 |\sum_{k}X_{jr}b^{(j)}_{lk}|^{2}\Big)^q\big|\mathfrak M^{(j)}\}\Big)\le
 Cq^qv^{-3q}\lna^{\frac2{\varkappa}}(\im( m_n^{(j)}(z)+\frac{y-1}z))^q.
\end{align}
Combining inequalities \eqref{eta21}, \eqref{eta22} and \eqref{bobkov}, we get
\begin{equation}
\E\{|\eta_{j2}|^{2q}\Big|\mathfrak M^{(j)}\}\le \frac{C^{2q}q^{3q}\lna^{\frac{4q}{\varkappa}}}{n^qv^{3q}}(\im (m_n^{(j)}(z)+\frac{y-1}z))^q.
\end{equation}
\end{proof}
\begin{lem}\label{nana}Under the conditions of  Theorem \ref{main} there exists a constant $C>0$ depending on $\varkappa$ and $y$ only that, for all $z=u+iv\in\mathbb G$ , and for any $j=1,\ldots,n$,
\begin{equation}
\frac1{|b_n(z)|}\le \frac C{|b_n^{(j)}(z)|}.
\end{equation}
\end{lem}
\begin{proof}
First we note
\begin{equation}\label{trup}
b_n(z)-b_n^{(j)}(z)=y(m_n(z)-m_n^{(j)}(z))=\frac y{2n}(\Tr\mathbf R-\Tr\mathbf R^{(j)})+\frac1{2nz}.
\end{equation}
This implies that
\begin{equation}
|b_n(z)-b_n^{(j)}(z)|\le \frac C{nv}\text{ a. s.}
\end{equation}
Furthermore, for $z\in\mathbb G$,
\begin{equation}
|b_n(z)|\ge \im(z+\frac{y-1}z+y s_y(z))\ge c|(z+\frac{y-1}z)^2-4y|^{\frac12}.
\end{equation}
It is straightforward to check that, for $z\in\mathbb G$,
\begin{equation}
|(z+\frac{y-1}z)^2-4y|^{\frac12}\ge \frac{\sqrt{\gamma(z)}}{|z|}.
\end{equation}
This yields , for $z\in\mathbb G$,
\begin{equation}
|b_n(z)|^{-1}\le \frac C{\sqrt{\gamma(z)}},
\end{equation}
and
\begin{align}
|b_n^{-1}(z)-{b_n^{(j)}}^{-1}(z)|\le |b_n^{-1}(z){b_n^{(j)}}^{-1}(z)||b_n^{-1}(z)-{b_n^{(j)}}(z)|\le 
\frac C{nv\sqrt{\gamma(z)}}|{b_n^{(j)}}^{-1}(z)|\le C|{b_n^{(j)}}^{-1}(z)|.
\end{align}
The last inequality concludes the proof of Lemma \ref{nana}

\end{proof}
\begin{lem}\label{bn}Under the conditions of  Theorem \ref{main} there exists a constant $C>0$ depending on $\varkappa$ and $y$ only that, for all $z=u+iv\in\mathbb G$ , and for any $j=1,\ldots,n$,
\begin{equation}
\E\Big\{\frac{|b_n(z)-b_n^{(j)}(z)|^{2q}}{|b_n(z)|^{2q}}\Big\}\le\frac{C^qq^{2q}\lna^{2q}}{(nv)^{2q}}.
\end{equation}
\end{lem}
\begin{proof}First we apply Lemma \ref{nana}. We obtain
\begin{equation}
\E\Big\{\frac{|b_n(z)-b_n^{(j)}(z)|^{2q}}{|b_n(z)|^{2q}}\Big\}\le\E\Big\{\frac{|b_n(z)-b_n^{(j)}(z)|^{2q}}{|b_n^{(j)}(z)|^{2q}}\Big\}
\end{equation}
Using equality \eqref{trup}, we get
\begin{equation}
\E\Big\{\frac{|b_n(z)-b_n^{(j)}(z)|^{2q}}{|b_n(z)|^{2q}}\Big\}
\le C^q\E\Big\{\frac{|\frac1n(\Tr\mathbf R(z)-\Tr\mathbf R^{(j)}(z))|^{2q}}{|b_n^{(j)}(z)|^{2q}}\Big\}
+\frac {C^q(1-y)^{2q}}{n^{2q}|z|^{2q}|b_n^{(j)}(z)|^{2q}}.
\end{equation}
Note that
\begin{equation}
|z||b_n^{(j)}(z)|\ge (1-y)|z|\frac v{|z|^2}.
\end{equation}
This implies that
\begin{equation}
\frac {C^q(1-y)^{2q}}{n^{2q}|z|^{2q}|b_n^{(j)}(z)|^{2q}}\le \frac {C^q}{n^{2q}v^{2q}}
\end{equation}
We use now equality
\begin{equation}
\Tr\mathbf R-\Tr\mathbf R^{(j)}=(1+\eta_j)R_{jj}.
\end{equation}
We get
\begin{align}
\E\Big\{\frac{|\frac1n(\Tr\mathbf R(z)-\Tr\mathbf R^{(j)}(z))|^{2q}}{|b_n^{(j)}(z)|^{2q}}\Big\}&\le 
\frac{C^q}{n^{2q}v^{2q}}\E|R_{jj}|^{2q}+
\E\Big\{\frac{|\frac1n\eta_jR_{jj}|^{2q}}{|b_n^{(j)}(z)|^{2q}}\Big\}.
\end{align}
Applying Cauchy - Schwartz inequality, we arrive
\begin{align}
\E\Big\{\frac{|\frac1n(\Tr\mathbf R(z)-\Tr\mathbf R^{(j)}(z))|^{2q}}{|b_n^{(j)}(z)|^{2q}}\Big\}&\le 
\frac{C^q}{n^{2q}v^{2q}}\E|R_{jj}|^{2q}
+\frac1{n^{2q}}\E^{\frac12}\Big\{\frac{|\eta_j|^{4q}}{|b_n^{(j)}(z)|^{4q}}\Big\}
\E^{\frac12}|R_{jj}|^{4q}.
\end{align}
Conditioning on $\mathfrak M^{(j)}$ and applying Lemma \ref{eta} , we get
\begin{align}
\E\Big\{\frac{|\frac1n(\Tr\mathbf R(z)-\Tr\mathbf R^{(j)}(z))|^{2q}}{|b_n^{(j)}(z)|^{2q}}\Big\}&\le 
\frac{C^q}{n^{2q}v^{2q}}\E|R_{jj}|^{2q}+\frac1{n^{2q}}\E^{\frac12}\Big\{\frac{C^q(\im m_n^{(j)}(z))^{4q}}
{v^{4q}|b_n^{(j)}(z)|^{4q}}\Big\}|R_{jj}|^{2q}\notag\\&
+\frac1{n^{2q}}\frac{C^q}\E^{\frac12}\Big\{\frac{C^{4q}(1-y)^{4q}}{|z|^{4q}|b_n^{(j)}(z)|^{4q}}\Big\}
\E^{\frac12}|R_{jj}|^{4q}\notag\\&+
\frac1{n^{2q}}\E^{\frac12}\Big\{\frac{|\im m_n^{(j)}(z)|^{2q}}{n^{2q}v^{6q}|b_n^{(j)}(z)|^{4q}}\Big\}
\E^{\frac12}|R_{jj}|^{4q}.
\end{align}
Using that $|b_n^{(j)}|\ge\im m_n^{(j)}$, $|zb_n^{(j)}|\ge \frac{(1-y)v}{|z|}$,  and $|b_n^{(j_)}|\ge c|(z+\frac{y-1}z)^2-4y|^{\frac12}$, for $z\in\mathbb G$, we get
\begin{align}
\E\Big\{\frac{|\frac1n(\Tr\mathbf R(z)-\Tr\mathbf R^{(j)}(z))|^{2q}}{|b_n^{(j)}(z)|^{2q}}\Big\}&\le 
\frac{C^q}{n^{2q}v^{2q}}\E|R_{jj}|^{2q}+\frac {C^q}{n^{2q}v^{2q}}\E|R_{jj}|^{2q}\notag\\&
+\frac {C^q}{n^{3q}v^{3q}}|(z+\frac{y-1}z)^2-4y|^{\frac q2}\E^{\frac12}|R_{jj}|^{4q}.
\end{align}
Applying now that, for $z\in\mathbb G$,
\begin{equation}
nv|(z+\frac{y-1}z)^2-4y|^{\frac12}\ge c,
\end{equation}
and according to Proposition \ref{rjj8},
\begin{equation}
\E|R_{jj}|^{4q}\le C^q,
\end{equation}
we obtain
\begin{align}
\E\Big\{\frac{|\frac1n(\Tr\mathbf R(z)-\Tr\mathbf R^{(j)}(z))|^{2q}}{|b_n^{(j)}(z)|^{2q}}\Big\}&\le 
\frac{C^q}{n^{2q}v^{2q}}.
\end{align}
Thus lemma \ref{bn} is proved.
\end{proof}


\subsection{The proof of Lemma \ref{trunc}}\label{proofoftrunc}\begin{proof} First we consider $m_n(z)-\widehat m_n(z)$.
 Denote by
 \begin{equation}
  \widehat{\mathbf R}=(\widehat V-z\mathbf I)^{-1}.
 \end{equation}
We have
\begin{equation}
 m_n(z)-\widehat m_n(z)= \frac1n\Tr \mathbf R(\mathbf V-\widehat{\mathbf V})\widehat{\mathbf R}.
\end{equation}
This representation and inequality $\max\{\|\mathbf R\|,\|\widehat{\mathbf R}\|\}\le v^{-1}$ imply
\begin{equation}
 |m_n(z)-\widehat m_n(z)|\le \frac 1{\sqrt nv^2}\|\mathbf W-\widehat{\mathbf X}\|_2=v^{-2}\left(\frac1{np}\sum_{j,l=1}^n|X_{jl}-\widehat X_{jl}|^2\right)^{\frac12}.
\end{equation}
From here it follows that
\begin{equation}\label{trunc1}
 \Pr\{|m_n(z)-\widehat m_n(z)|>\frac{C}{n^2v^2}\}\le \sum_{j,l=1}^n\Pr\{|X_{jl}-\widehat X_{jl}|>\frac{C}{n^2}\}.
\end{equation}
Note that
\begin{equation}
 X_{jl}-\widehat X_{jl}=X_{jl}\mathbb I\{|X_{jl}|\ge C\lna^{\frac1{\varkappa}}\}-\E X_{jl}\mathbb I\{|X_{jl}|\ge C\lna^{\frac1{\varkappa}}\}.
\end{equation}
Condition \eqref{exptails} implies that
\begin{equation}
 |\E X_{jl}\mathbb I\{|X_{jl}|\ge C\lna^{\frac1{\varkappa}}\}|\le \exp\{-c\lna\}\le \frac {C}{2n^2}.
\end{equation}
From here it follows that
\begin{equation}\label{trunc3}
 \Pr\{|X_{jl}-\widehat X_{jl}|>\frac{C}{n^2}\}\le \Pr\{|X_{jl}|\ge C\lna^{\frac1{\varkappa}}\}\le A\exp\{-c\lna\}.
\end{equation}
Inequalities \eqref{trunc1} and \eqref{trunc3} together imply that there exists a constant $c'$ such that
\begin{equation}\label{trunc01}
 \Pr\{|m_n(z)-\widehat m_n(z)|>\frac{C}{n^2v^2}\}\le \exp\{-c'\lna\}.
\end{equation}

We prove now that
\begin{equation}\label{trunc4}
 \Pr\{|\widetilde m_n(z)-\widehat m_n(z)|\ge \frac C{n^2v^2}\}\le \exp\{-c\lna\}.
\end{equation}
Repeating the arguments of \eqref{trunc} -- \eqref{trunc3}, we need to prove
\begin{equation}\label{trunc5}
 \Pr\{|\widehat X_{jk}-\widetilde X_{jk}|>\frac C{n^3}\}\le \Pr\{(1-\sigma_{jk})\sigma_{jk}^{-1}
 |\widehat X_{jk}|>\frac C{n^2}\}.
\end{equation}
Note that
\begin{equation}
 \sigma_{jk}^2=1-\E X_{jk}^2\mathbb I\{|X_{jk}|\ge c\lna^{\frac1{\varkappa}}\}-
 (\E X_{jk}\mathbb I\{|X_{jk}|\ge c\lna^{\frac1{\varkappa}}\})^2\ge 1-\exp\{-c'\lna\}.
\end{equation}
The last bound is uniform in $j,k=1,\ldots,n$.
This implies that
\begin{equation}\label{trunc6}
 (1-\sigma_{jk})\sigma_{jk}^{-1}\le \exp\{-c''\lna\}.
\end{equation}
Inequalities \eqref{trunc5} and \eqref{trunc6} together imply \eqref{trunc4}.
Thus Lemma \ref{trunc} is proved.
\end{proof}
\subsection{The proof of Lemma \ref{key}}\label{profofkey} Note that by Lemma \ref{stmp}
\begin{equation}
|S_y(z)|\le y^{-\frac12}\text{ and }|z+\frac{y-1}z+yS_y(z)|\ge \sqrt y.
\end{equation}
This implies that for any $z=u+iv$ with $v\ge v_k$ and for any $\omega\in\mathcal A_k\cap\mathcal B_k$
\begin{equation}\label{ap1}
|m_n(z)|\le \frac3{2\sqrt y}\text{ and }|a_n(z)|\ge\frac{\sqrt y}2.
\end{equation}
The relations \eqref{repr01} and \eqref{ap1} together imply, for $j=1,\ldots,n$,
\begin{equation}
|R_{jj}|\le 2y^{-\frac12}(1-2\gamma_0)^{-1}.
\end{equation}
The last inequality yields 
\begin{equation}\label{ap2}
|T_n(z)|\le 2(1-2\gamma_0)^{-1}\gamma_0.
\end{equation}

Assume that $|b_n(z)|\ge \sqrt{{|\delta_n(z)|}}$. Then using \eqref{lambda'} we get
\begin{equation}
|\Lambda_n(z)|\le \sqrt{|T_n(z)|}.
\end{equation}
If $|b_n(z)|\le \sqrt{{|T_n(z)|}}$, we assume first that
\begin{equation}\label{asum1}
|\Lambda_n(z)|> 2\sqrt{|T_n(z)|}.
\end{equation}
Applying triangle inequality, we may write
\begin{equation}
|z+\frac{y-1}z+2ys_y(z)|\ge |\Lambda_n(z)|-|b_n(z)|\ge \sqrt{|T_n(z)|}.
\end{equation}
From the other hand, according to Lemmas \ref{sign} and \ref{imsqrt},
\begin{align}
|b_n(z)|&\ge \im\{z+\frac{y-1}z+ys_y(z)\}\ge \notag\\&\ge\frac 12
|(z+\frac{y-1}z)^2-4y|^{\frac12}=\frac 12|z+\frac{y-1}z+2ys_y(z)|\ge 
\frac 12\sqrt{|T_n(z)|}.
\end{align}
The last inequality and relation \eqref{lambda'} together imply
\begin{equation}\label{del1}
|\Lambda_n(z)|\le 2\sqrt{|T_n(z)|}.
\end{equation}

The inequalities \eqref{del1}contradicts to the asumption  \eqref{asum1}. This implies that  for $\omega\in \mathcal A_k\cap\mathcal B_k$
\begin{equation}
|\Lambda_n(z)|\le 2\sqrt{|T_n(z)|}\le \frac{2\sqrt {2\gamma_0}}{(1-2\gamma_0)}.
\end{equation}
Since $y\in(0,1)$ we may write
\begin{equation}
|\Lambda_n(z)|\le \frac{2\sqrt {2\gamma_0}}{\sqrt{1-2\gamma_0}\sqrt y}.
\end{equation}
For $\gamma_0\le \frac1{130}$ we have for $\omega\in \mathcal A_k\cap\mathcal B_k$
\begin{equation}\label{f1}
|\Lambda_n(z)|\le \frac{1}{4\sqrt y}.
\end{equation}
Note that
\begin{equation}
|\Lambda_n'(z)|\le \frac2{v^2},
\end{equation}
and
\begin{equation}
v_k-v_{k-1}=\frac{\sqrt y}{4n^2}.
\end{equation}
From this relation using Teilor's formula we get, for any $v_{k-1}\ge v\ge v_k$, and for $c_0\ge 2$,
\begin{equation}\label{f2}
|\Lambda_n(u+iv_{k-1})-\Lambda_n(u+iv)|\le \frac{4\sqrt y}{4n^2v_k^2}\le 
\frac{1}{\beta_n^8c_0^2\sqrt y}\le \frac1{4\sqrt y}.
\end{equation}
Inequalities \eqref{f1} and \eqref{f2} together imply for $\omega\in\mathcal A_{k-1}\cap\mathcal B_{k-1}$ and for any $z=u+iv$ with $v\in[v_k,v_{k-1}]$,
\begin{equation}
|\Lambda_n(u+iv)|\le \frac1{2\sqrt y}.
\end{equation}
Thus Lemma \ref{key} is proved.



\vskip 0.05cm
{\bf Acknowledgement.} The authors would like to thank S. G. Bobkov for
drawing their attention to some references about large deviations for martingales and quadratic forms.



\end{document}